\documentclass{elsarticle}

\usepackage{oppapers2}
\usepackage{p1cites}

\usepackage{tikz}
\usetikzlibrary{positioning}

\usepackage[matrix,cmtip,arrow,curve]{xy}

\usepackage[unicode]{hyperref}

\newcommand{\inthom}{\operatorname{map}}
\newcommand{\sethom}{\hom}

\newcommand{\bk}{{\bar{k}}}
\newcommand{\bj}{{\bar{j}}}
\newcommand{\shuff}[1]{\mathcal{S}_{#1}}

\newcommand{\zh}{Z_{h\pi}}
\newcommand{\bw}{\overline{W}}
\newcommand{\coact}{\varrho}
\newcommand{\slush}[1]{\underline{#1}}
\newcommand{\inta}{{\zeta}}
\newcommand{\intalg}{\inta_{alg}}
\newcommand{\intb}{{\chi}}
\newcommand{\faked}{a}
\newcommand{\mess}{\ell}
\newcommand{\abut}{\iota}

\renewcommand{\exte}{u}

\newcommand{\algr}{\nu}
\newcommand{\booeta}{\epsilon} %CHANGE ME

\begin{document}
\begin{frontmatter}

\title{Spectral sequence operations converge to Araki-Kudo operations}
\author{Philip Hackney\fnref{fn}}
\ead{hackney@math.ucr.edu}
\address{Department of Mathematics, University of California, Riverside, 900 University Avenue,
Riverside, CA 92521, USA}
\fntext[fn]{Phone: 951-827-5402, Fax: 951-827-7314}

\begin{abstract} 
Previously we constructed operations in the mod 2 homology spectral sequence associated to a cosimplicial $E_\infty$-space $X$. The correct target for this spectral sequence is the homology of $\Tot X$. Noting that in this setting  $\Tot X$ is an $E_\infty$-space, we show that our operations agree with the usual Araki-Kudo operations in the target. We also prove that the multiplication in the spectral sequence agrees with the multiplication in $H_*(\Tot X)$.

\end{abstract}


\end{frontmatter}

\section{Introduction}

Let $\msc$ be a fixed $E_\infty$ operad. In \cite{me1}, we gave a new construction of operations in the mod 2 homology spectral sequence of a cosimplicial $\msc$-space $X$. They are formally similar to (and probably equal to) those given by Turner in \cite{turner}, and take the form of `vertical' and `horizontal' operations:
\begin{align*}
Q^{m}:E^r_{-s,t} &\to E^r_{-s,m+t} & m&\geq t \\
Q^m:E^r_{-s,t} &\to E^w_{m-s-t,2t} & m&\in[t-s,t]\end{align*}
where $w\in [r,2r-2]$. The precise value of $w$ is not relevant for the current paper, since we are concerned with what happens in the limit $r=\infty$, and, in this setting, the operations may be given a unified form 
\[ Q^m: E^\infty_{-s,t} \to E^\infty_{-v,v-s+m+t} \]
where \begin{equation} v(t,s,m) = \begin{cases} s & m\geq t \\ t+s-m & t-s \leq m \leq t. \end{cases} \label{E:v}\end{equation}

In this context, $\Tot(X)$ is a $\msc$-space, which implies that the mod 2 homology $H_*(\Tot X)$ has Araki-Kudo operations. We can compare the operations on page $E^\infty$ of the spectral sequence with the usual Araki-Kudo operations using Bousfield's identification of $H_*(\Tot X)$ as the abutment of the homology spectral sequence (see \cite{bousfield}).

\begin{thm}\label{T:internal}
Suppose that $X$ is a cosimplicial $\msc$-space. Then the Araki-Kudo operation $Q^m$ has the following effect on the filtration:
\[ Q^m [ F^{-s} H_{t-s} (\Tot X)] \ci F^{-v}H_{t-s+m} (\Tot X), \]
where $v$ is as in \eqref{E:v}. 
Furthermore, for each $m$  the following diagram commutes:
\[ \xymatrix{
\left({F^{-s}}/{F^{-s-1}}\right) H_{t-s} (\Tot X) \ar@{->}[r]^-\abut \ar@{->}[d]^{Q^m} & E_{-s,t}^\infty (X) \ar@{->}[d]^{\ext^m}\\
\left( {F^{-v}}/{F^{-v-1}}\right) H_{t-s+m} (\Tot X) \ar@{->}[r]^-\abut & E_{-v, v-s+m+t}^\infty (X),
}\]
where the map on the right is the one developed in \cite{me1}.
\end{thm}

The first statement along these lines appeared in \cite{turner}.
We also prove a similar statement relating the multiplication in the spectral sequence to the multiplication on $H_*(\Tot X)$ (Corollary~\ref{C:mult}). 
The results of this paper are actually slightly more general: they apply to external operations and external multiplication for arbitrary cosimplicial spaces. In particular, Theorem~\ref{T:internal} follows directly from our main theorem, \ref{T:convergence}, which is the external analog of \cite[5.11]{turner}. The reason for the present paper is not that the external version of this theorem is much harder to prove than the internal version, but rather that the sketch of a proof given in \cite{turner} turns out to need a lot more work (see section~\ref{S:turner}). This article is devoted to providing a complete proof.

Only minor changes are necessary to extend the results of this paper to the case when $X$ is a cosimplicial $E_{n+1}$ space. This is discussed in \cite{me2b}.

\subsection{Notation and Background}\label{S:notationbackground} We work over the field $\K=\Z/2$, and all modules should be interpreted as being $\K$-vector spaces unless otherwise mentioned. Let $\pi := \set{e, \sigma}$ be the group with two elements. For a set $Z$ with basepoint $*$, the notation $\K Z$ will always mean the free $\K$-module with basis $Z\setminus \set{*}$, so that $\K$ defines a functor $\Set_* \to \K\operatorname{Mod}$. If $Z$ is just a set, then $\K Z$ will be denote the free $\K$-module with basis $Z$.

In this paper, we take ``space'' to mean ``simplicial set.'' Important examples of simplicial sets are the simplices $\Delta^p = \hom_{\Delta}( -, [p])$, the classifying space $B\pi$, and the contractible space $E\pi$. The models we use for the latter two spaces are those from \cite[V.4]{gj}, with $B\pi_q = \pi^{\times q}$ and $E\pi_q = \pi^{\times (q+1)}$, although this specificity plays a role only beginning in section~\ref{S:comodules}.

\subsection{Normalization, conormalization, and the spectral sequence}

If $\msa$ is an abelian category, we write $N: \msa^{\Delta^{op}} \to \ch_{\geq 0}(\msa)$ for the normalization functor. 
Specifically, if $A\in \msa^{\Delta^{op}}$, we set \begin{align*}
N_p(A) &= \coker \left(\bigoplus_{k=0}^{p-1} s_k: \bigoplus A_{p-1} \to A_p \right) \\
\partial(a_p) &= \sum_{k=0}^{p} (-1)^k d_k(a).
\end{align*}
We also write 
\[ S_*: \operatorname{Spaces} = \Set^{\Delta^{op}} \overset{\K}{\to} (\K\operatorname{Mod})^{\Delta^{op}} \overset{N}{\to} \ch \]
for the mod 2 normalized chains functor, where $\ch = \ch_{\ge 0}(\K\operatorname{Mod})$ is the category of nonnegatively graded chain complexes over $\K$. A recurring example is $S_*(E\pi)$, which, using the model for $E\pi$ from the previous paragraph, is isomorphic to the complex $W$ with
\begin{align*} W_i &= \begin{cases} e_i \cdot \K \pi & i\geq 0 \\ 0 & i <0 \end{cases}\\
d(e_i) &= (1+\sigma) e_{i-1} \end{align*}
where $\K \pi$ is the group algebra.

\begin{defn}
Suppose that $A$ and $B$ are simplicial $\K$-modules. The tensor product is the simplicial $\K$-module with $(A\ten B)_n = A_n \ten B_n$, where $\ten = \ten_\K$. 
We will regularly use the following two natural transformations (see \cite[Corollaries VIII.8.6 and VIII.8.9]{homology}). 
The first is the 
Alexander-Whitney map
\begin{gather*}
AW:N(A\ten B) \to N(A) \ten N(B) \\
AW(a_n\ten b_n) = \sum_{p+q=n} d_{p+1} \cdots d_n a_n \ten d_0 \cdots d_{p-1} b_n
\end{gather*}
and the second is the 
shuffle map
\begin{gather*} \nabla: N(A) \ten N(B) \to N(A\ten B) \\
\nabla(a_p\ten b_q) = \sum_{\substack{ (p,q)-\text{shuffles}\\ \tau}} s_{\tau(p+q-1)} \cdots s_{\tau(p)} a_p \ten s_{\tau(p-1)} \cdots s_{\tau(0)} b_q,  \end{gather*}
where we consider $(p,q)$-shuffles as permutations of the set \[ \set{0,1,\dots, p+q-1}. \]
\end{defn}
A consequence of the Eilenberg-Zilber theorem and the K\"unneth theorem is that $AW$ and $\nabla$ are inverse chain homotopy equivalences.

Again assuming that $\msa$ is an abelian category, we write $C: \msa^\Delta \to \coch^{\geq 0} (\msa)$ for conormalization, with the convention that
\begin{equation} CY^p = \coker \left(\bigoplus_{k=1}^p d^k: \bigoplus Y^{p-1} \to Y^p \right). \label{E:conorm}\end{equation} 
If $\msa$ happens to be the category $\ch$ of chain complexes over $\K$, then we interpret $C$ as landing in the category of (second-quadrant) bicomplexes.
If $Y$ is a cosimplicial chain complex, the indexing is given by
\[ C(Y)_{-p,q} = C(Y_q)^p. \]

We will also need the cosimplicial versions of the Alexander-Whitney and shuffle maps. When we define them for cosimplicial chain complexes $Y$ and $Z$, we interpret them as maps of bicomplexes between $C(Y) \ten C(Z)$ and $C(Y\ten Z)$ by \cite[Lemma 4.1]{me1}.
\begin{defn}[{\cite[Appendix]{bauesmuro}}]\label{D:awshuff}
The Alexander-Whitney map $AW$ is defined on $C(Y)^p \ten C(Z)^q$ by
\[ AW(y^p\ten z^q) = d^{p+q} \cdots d^{p+1} y^p \ten d^{p-1} \cdots d^{0} z^q. \] The shuffle map $\nabla$ is defined on $C(Y\ten Z)^n$ by
\[ \nabla(y^n \ten z^n) = \sum_{p+q=n} \sum_{\substack{ (p,q)-\text{shuffles}\\ \tau}} s^{\tau(p)} \cdots s^{\tau(p+q-1)} y^n \ten s^{\tau(0)} \cdots s^{\tau(p-1)} z^n. \]

\end{defn}

Given a bicomplex $B$, we will let $TB$ denote the product total complex
\begin{equation} TB_m = \prod_{j} B_{j,m-j} \label{E:total}\end{equation} together with the filtration by columns:
\[ F^{k}_m = \prod_{j\leq k} B_{j,m-j}. \] 
We will write 
\begin{equation}
 T_\mess B := TB / F^{-(\mess + 1)}  \label{E:totalell} \end{equation} for the canonical quotients. The spectral sequences used in this paper are derived from this filtration. In particular, the spectral sequence associated to a cosimplicial chain complex $Y$ is, by definition, the spectral sequence associated to the above filtration for the bicomplex $B=CY$. Furthermore, the spectral sequence associated to a cosimplicial space $X$ is defined as the spectral sequence associated to the cosimplicial chain complex $S_* (X)$. 

\begin{notation}
If $C$ is a filtered chain complex with $F^{-s-1} \ci F^{-s}$, we will write 
\begin{align*}
Z_{-s,t}^r &= \setm{x\in F^{-s} C_{t-s}}{\partial x \in F^{-s-r} C} \\
B_{-s,t}^r &= \partial Z^{r-1}_{-s+r-1, t-r+2} + Z^{r-1}_{-s-1, t+1}
\end{align*} for the $r$-cycles and $r$-boundaries, and
$E_{-s,t}^r = Z_{-s,t}^r / B_{-s,t}^r$ for the $(-s,t)$ position of the $r^{\text{th}}$ page of the spectral sequence . 
\end{notation}

\subsection{Totalization and the abutment of the spectral sequence}\label{S:totalization}

We now give an indication as to why this spectral sequence approaches $H_*(\Tot X)$, following \cite{bousfield}. 
For a cosimplicial space $X$ and $\mess\leq \infty$ we have
\begin{equation} \Tot_\mess X := \inthom (\sk_\mess \Delta, X) \ci \prod_p \inthom(\sk_\mess \Delta^p, X^p). \label{E:Totell} \end{equation}
Here, for spaces $Z$ and $Z'$, $\inthom (Z,Z')$ is the mapping space with $\inthom(Z,Z')_q = \sethom(Z\times \Delta^q, Z')$. A $q$-simplex of $\Tot_\mess X$ should thus be taken as a sequence $(f^\bullet)=f$ with
\[ f^p: \sk_\mess \Delta^p \times \Delta^q \to X^p \] a map of spaces so that the diagram
\[ \xymatrix{
\sk_\mess \Delta^p \times \Delta^q \ar@{->}[r]^-{f^p} \ar@{->}[d]_{\sk_\mess \Delta^\alpha \times \Delta^q} & X^p \ar@{->}[d]^{X^\alpha} \\
\sk_\mess \Delta^{\bar{p}} \times \Delta^q \ar@{->}[r]^-{f^{\bar{p}}}  & X^{\bar{p}}
}\]
commutes for each $\alpha: [p] \to [\bar{p}]$.

We remark that if $V$ is a cosimplicial simplicial $\K$-module, then $(\Tot_\mess V)_q$ may be identified with $\hom(\K \sk_\mess \Delta \ten \K \Delta^q, V)$ using adjointness.
Considering the case $V=\K X$ for a cosimplicial space $X$,
there is a map
\begin{equation}
\begin{aligned} \Tot_\mess X &\to \Tot_\mess \K X \\
f&\mapsto \K f
\end{aligned}
\label{E:geominterchange} \end{equation}
 where
\[ \K f^p: \K \sk_\mess \Delta^p \ten \K \Delta^q = \K (\sk_\mess \Delta^p \times \Delta^q) \to \K X^p \]
is the induced map. 
This defines a natural transformation
\begin{equation}
\algr: \K \Tot X \to \Tot \K X \label{E:algr}
\end{equation}
which will appear frequently in later sections.
If $f \in (\Tot_\mess V)_q$, we will write \begin{equation} \slush{f} = T_\mess CN f : T_\mess CN (\K \sk_\mess \Delta^\bullet \ten \K \Delta^q) \to T_\mess CN V. \label{E:funderline} \end{equation}

We now recall, for $\mess\leq \infty$, the quasi-isomorphism \[ \phi_\mess: N \Tot_\mess V \to T_\mess CN V \]  from \cite[Lemma 2.2]{bousfield}.
Let 
\[ 0 \neq \imath_k \in N_k ( \K \Delta^k) \cong \K \]
be the class representing the top dimensional simplex of $\Delta^k$, and let 
 $\imath_k \wedge \imath_q$ be the image of 
\[ \imath_k \ten \imath_q \in N_k(\K \Delta^k) \ten N_q(\K \Delta^q) = N_k(\K \sk_\mess \Delta^k) \ten N_q(\K \Delta^q)\] 
under the shuffle map
\[ N(\K \sk_\mess \Delta^k) \ten N(\K \Delta^q) \overset{\nabla}{\to} %N[ \K \Delta^k \ten \K \Delta^q ] = 
N[ \K \sk_\mess \Delta^k \ten \K \Delta^q ] \hookrightarrow T_\mess CN [ \K \sk_\mess \Delta^\bullet \ten \K \Delta^q ] .\]
The map $\phi_\mess$ is defined by  \begin{equation} \phi_\mess (f) = \slush{f} \left[ \prod_{k=0}^\mess \imath_k \wedge \imath_q \right], \label{E:phiell} \end{equation}
and observe that 
\[ \phi_\mess(f) = \prod_{k=0}^\mess \slush{f}(\imath_k \wedge \imath_q) \] for $\mess \leq \infty$.

The composite
\[ S_* \Tot X = N \K \Tot X \to N \Tot (\K X) \overset{\phi_\infty}{\to} T CN \K X \] 
is used to define the filtration on $H_*(\Tot X)$. The filtration is given by
\[ F^{-s} = \ker(H \Tot X \to H TCN\K X \to H T_{s-1} CN\K X ),  \]
so $F^{-(s+1)} \ci F^{-s}$. This gives the abutment map
\begin{equation} \abut: (F^{-s} / F^{-s-1}) H (\Tot X) \hookrightarrow E_{-s}^\infty (X; \K) \label{SYMBOL:ABUTMENT} \end{equation}
from \cite[p.~364]{bousfield}.

\subsection{Totalization, products, and operad actions}\label{S:tpoa}

In this section we mention how the totalization of a cosimplicial $\msc$-space should be regarded as a $\msc$-space. Note 
that this section generalizes to other operads.

Suppose that $X$ and $X'$ are cosimplicial spaces and $\mess \leq \infty$. Using the diagonal \[ \sk_\mess \Delta^\bullet \times \Delta^q \to (\sk_\mess \Delta^\bullet \times \Delta^q) \times (\sk_\mess \Delta^\bullet \times \Delta^q),\] we have a function
\[ \Tot_\mess (X)_q \times \Tot_\mess(X')_q \to \Tot_\mess (X\times X')_q.\]
Since $\Tot_\mess(-)_q$ is a mapping space $\inthom(\sk_\mess \Delta^\bullet\times \Delta^q, -)$, this actually gives an isomorphism
\[ \Tot_\mess(X) \times \Tot_\mess(X') \cong \Tot_\mess (X\times X'). \]

If $Z$ is a space, we may consider the constant cosimplicial space $cZ$ with $cZ^p = Z$ for all $p$ and notice that $\Tot (cZ) = Z$. 
For a cosimplicial space $X$, we then have a composite of $\Sigma_n$-equivariant isomorphisms 
\begin{align*} \msc(n) \times \Tot(X)^{\times n} \to \Tot(c\,\msc(n)) \times \Tot(X)^{\times n} &\to \Tot( c\,\msc(n) \times X^{\times n}) \\ & = \Tot (\msc(n) \times X^{\times n}). \end{align*}
Now suppose that $X$ is a cosimplicial $\msc$-space. To define the $\msc$ structure on $\Tot(X)$, combine the above isomorphism with the totalization of the map of cosimplicial spaces
$ \msc(n) \times X^{\times n} \to X$
to obtain the map
\[ \msc(n) \times \Tot(X)^{\times n} \to \Tot (\msc(n) \times X^{\times n}) \to \Tot(X). \]
One can check that these structure maps make $\Tot(X)$ into an algebra over $\msc$.

\subsection{Araki-Kudo operations}\label{S:arakikudo}

There are two steps needed to define, in the style of \cite{may}, the usual Araki-Kudo operations for a $\msc$-space $Z$. %The first step is to define, 
For any chain complex $C$, we have graded (non-additive) functions of degree $m$ 
\[ q^m: C \to W\tp (C\ten C), \] 
defined the formula
\[ c \mapsto e_{m-|c|} \ten c \ten c + e_{m-|c|+1} \ten c \ten \partial c, \]
which we call \emph{external operations}. 
Notice that $q^m \partial = \partial q^m$ and that $q^m$ induces a homomorphism in homology.
When $C=S_*(Z)$, we have the chains of the $\msc(2)$-structure map
\[ S_* (\msc(2) \times_\pi (Z \times Z) ) \to S_*(Z), \]
which we connect to the external operations via
\[ W\tp (S_*(Z) \ten S_*(Z)) \to S_*(\msc(2)) \tp (S_*(Z) \ten S_*(Z)) \overset\nabla\to S_*( \msc(2) \times_\pi (Z\times Z)). \]
This composite is a quasi-isomorphism since  $S_*(\msc(2)) \simeq_\pi W$ and the shuffle map $\nabla$ is a $\pi$-equivariant quasi-isomorphism.
The composite
\[ H_*(Z) \overset{q^m}{\longrightarrow} H_{*+m} (\msc(2) \times_\pi (Z\times Z)) \to H_{*+m} (Z) \] is then an Araki-Kudo operation. The original paper \cite{kudoaraki} does not mention that the operations factor through the homology of the homotopy orbits of $Z\times Z$, but this is how they are constructed in \cite{may}.

\subsection{Bousfield-Kan universal examples}\label{S:universality}

\newcommand{\cof}{\operatorname{cof}}

We will need certain small cosimplicial spaces later in the paper.
The Bousfield-Kan universal examples (introduced in \cite{bk}) are cosimplicial simplicial pointed sets defined, for $t\geq s$, by
\[ \bkspace_{(\infty,s,t)} := \Sigma^{t-s} \cof \left( \sk_{s-1} \Delta_+^\bullet \to  \Delta_+^\bullet \right), \]
where $\Sigma$ is the Kan suspension and $\Delta_+^\bullet$ is obtained by adding a disjoint basepoint to %each cosimplicial degree of 
the standard cosimplicial simplicial set $\Delta^\bullet$. 
The $E^1$ page of the spectral sequence associated to $\bkspace_{(\infty,s,t)}$ consists of a single non-zero class in bidegree $(-s,t)$.
% For $r\geq 2$, pages $E^2$ through $E^r$ of the spectral sequence associated to $\bkspace_{(r,s,t)}$ are given in figure~\ref{F:uebk}. 
If $V$ is a cosimplicial simplicial $\K$-module (the main class of examples are given by  $V=\K X$, where $X$ is a cosimplicial space) and $x$ is an infinite cycle with $[x] \in E_{-s,t}^\infty(V)$, then $x$ determines a map of cosimplicial simplicial modules $\K \bkspace_{(\infty,s,t)} \to V$ so that on $E^\infty$ the unique nonzero class $
\imath$ maps to   $ [x]. $
This universal property is given in \cite{bk} for $\Z$ instead of $\K$, or one may use \cite[Proposition \propuniversalprop]{me1} and the fact that normalization is an equivalence of categories by the Dold-Kan correspondence.
See also figure~\ref{F:inftyconormed}.

% \begin{figure}[ht]
% \centering
% \begin{tikzpicture}
%     %Axes
%     \draw[->] (0.2,0) -- (-4,0);
%     \draw[->] (0,-0.2) -- (0,3);
%     %Horizontal ticks
%     \draw (-1,2pt) -- (-1,-2pt) node[below] {$-s$};
%     \draw (-3,2pt) -- (-3,-2pt) node[below] {$-(s+r)$};
%     %Vertical Ticks
%     \draw (2pt,1) -- (-2pt,1) node[right] {$t$};
%     \draw (2pt,2) -- (-2pt,2) node[right] {$t+r-1$};
%     %Points
%     \fill (-1,1) circle (2pt) node[right] {$\imath$};
%     \fill (-3,2) circle (2pt) node[left] {$\ssd^r(\imath)$};
%     %Arrow
%     \draw [->] (-1.2,1.1) -- (-2.8,1.9);
% \end{tikzpicture}
% \caption{Spectral Sequence for $\bkspace_{(r,s,t)}$}\label{F:uebk}
% \end{figure}
% The important case for the present paper is when $r=\infty$, and in that case the s

\subsection{Spectral sequence operations}\label{S:sso} We briefly recall, from \cite{me1}, how to define external operations in the spectral sequence of a cosimplicial space $X$ at $E^\infty$. We actually show how to do it for a cosimplicial simplicial module $V$. The spectral sequence associated to $\K E\pi \tp (\K \bkspace_{(\infty, s, t)})^{\ten 2}$ has $E^2 = E^\infty$, and takes the form of figure~\ref{F:e2bkinf}, where each lattice point on a solid line is a copy of $\K$ and each other lattice point is zero \cite[Theorem 7.1]{me1}. 
\begin{figure}[ht] 
%\centering
% \begin{tikzpicture}
%     %Axes
%     \draw[->] (0.2,0) -- (-5,0);
%     \draw[->] (0,-0.2) -- (0,5);
%     %Horizontal ticks
%     \draw (-1,2pt) -- (-1,-2pt) node[below] {$-s$};
%     \draw (-4,2pt) -- (-4,-2pt) node[below] {$-2s$};
%     %Vertical Ticks
%     \draw (2pt,1) -- (-2pt,1) node[right] {$2t$};
%     %Points
% %    \fill (-4,1) circle (1.5pt);
% %    \fill (-1,1) circle (1.5pt);
%     %Arrow
%     \draw [very thick] (-4,1) -- (-1,1);
%     \draw [->,very thick] (-1,1) -- (-1,5);
% \end{tikzpicture}  
\centering \scalebox{0.75} {
\begin{tikzpicture}	
	\draw (2.2,0) -- (-5,0);		
    \draw[->] (2,-0.2) -- (2,5);
    \draw (-1,2pt) -- (-1,-2pt) node[below] {$-s$};
    \draw (-4,2pt) -- (-4,-2pt) node[below] {$-2s$};
    \draw (1.9,1) -- (2.1,1) node[right] {$2t$};
    \draw [very thick] (-4,1) -- (-1,1);
    \draw [->,very thick] (-1,1) -- (-1,5);
\end{tikzpicture} }
\caption{$E^2( \K E\pi \tp (\K \bkspace_{(\infty,s,t)} \ten \K\bkspace_{(\infty,s,t)}))$}\label{F:e2bkinf}
\end{figure}
Write $\exte_{p,q}$ for the nonzero element in bidegree $(p,q)$ when there is one. We consider these nonzero elements $\exte_{p,q}$ of $E^\infty(\K E\pi \tp (\K \bkspace_{(\infty, s, t)})^{\ten 2})$ as the images of the external operations applied the element $\imath$ in the spectral sequence for $\K \bkspace_{(\infty, s, t)}$. Now, if $V$ is another cosimplicial simplicial module and we are given an infinite cycle $v\in Z^\infty_{-s,t}(V)$,  let $R_v: \K \bkspace_{(\infty, s, t)} \to V$ be the representing map so that $E^\infty(R_v)(\imath) = [v]$. By naturality, we should define the images of the operations on $[v]$ as the elements \[ E^\infty (1 \ten R_v^{\ten 2}) (\exte_{p,q}) \in E^\infty(\K E\pi \tp V^{\ten 2}). \]
This is, in fact, how we define the operations in \cite{me1}. 
% \begin{rem}
% All external operations on $\imath \in E^\infty_{-s,t}(\bkspace_{(\infty, s,t)})$ are nontrivial.
% \end{rem}

\section{Main Theorem and Outline of Proof}

For the remainder of the paper we will work externally (see section~\ref{S:arakikudo}) and usually do not assume that any spaces are $\msc$-spaces.  For a space $Z$ and each $m$, there is an external operation \[ q^m: H_*(Z) \to H_{*+m}(E\pi \times_\pi Z^{\times 2}) \]
given on the chain level by
\begin{align*} S_* (Z) &\overset{q^m}\to W_* \tp (S_* (Z) \ten S_* (Z)) \to S_* (E\pi \times_\pi (Z\times Z)).
\end{align*}
In particular, if we have a cosimplicial space $X$ we can take $Z=\Tot X$ and so we have
\[ q^m: H_*(\Tot X) \to H_{*+m} (E\pi \times_\pi (\Tot X)^{\times 2}). \] Furthermore, there is an interchange map \eqref{E:interchange}
\[ \inta: E\pi \times_\pi (\Tot X)^{\times 2} \to \Tot (E\pi \times_\pi X^{\times 2}) \] which is often a homology isomorphism (for instance when $X$ is Reedy-fibrant -- a fact we shall not need and hence shall not prove here). Thus we consider the composite operations
\[ Q^m: H_* (\Tot X) \to H_{*+m} (\Tot (E\pi \times_\pi X^{\times 2})) \]
to be the classic external Araki-Kudo operations. Since $\inta$ often induces an isomorphism in homology, this is a minor variant of the constructions in \cite{may}.

The comparison will be written in terms of the abutment map \eqref{SYMBOL:ABUTMENT}, which is a monomorphism
\[ \abut: (F^{-s} / F^{-s-1}) H_{t-s} (\Tot X) \hookrightarrow E_{-s,t}^\infty (X; \K) = E_{-s,t}^\infty (S_*(X)). \]

\begin{thm}\label{T:convergence} Suppose that $X$ is a cosimplicial space. Then
\[ Q^m [ F^{-s} H_{t-s} (\Tot (X))] \ci F^{-v}H_{t-s+m} (\Tot(E\pi \times_\pi X^{\times 2})) \]
where \[ v(t,s,m) = \begin{cases} s & m\geq t \\ t+s-m & t-s \leq m \leq t. \end{cases} \]Furthermore, for each $m$  the following diagram commutes:
\[ \xymatrix{
\left({F^{-s}}/{F^{-s-1}}\right) H_{t-s} (\Tot X) \ar@{->}[r]^-\abut \ar@{->}[dd]^{Q^m} & E_{-s,t}^\infty (S_*(X)) \ar@{->}[d]^{\ext^m} \\
& E_{-v, v-s+m+t}^\infty (W\tp S_*(X)^{\ten 2}) \ar@{->}[d]^{\cong}_{E^1 - iso} \\
\left( {F^{-v}}/{F^{-v-1}}\right) H_{t-s+m} (\Tot (E\pi \times_\pi X^{\times 2})) \ar@{->}[r]^-\abut & E_{-v, v-s+m+t}^\infty (S_*(E\pi \times_\pi X^{\times 2})).
}\]
%where the horizontal maps are those given by Bousfield.
\end{thm}

% \begin{cor}
% Theorem~\ref{T:internal} holds.
% \end{cor}

The main theorem is then a corollary of this theorem.
\begin{proof}[Proof of Theorem~\ref{T:internal}]
This follows from the definition of the $\msc$-structure on $\Tot X$ given in section~\ref{S:tpoa}, naturality, and Theorem~\ref{T:convergence}.
\end{proof}

%This theorem implies theorem~\ref{T:internal} because we 

The remainder of the paper will be dedicated to the proof of Theorem~\ref{T:convergence}. Let us describe the  strategy. We will first define `algebraic' operations for a cosimplicial simplicial $\K$-module $V$, which are maps
\[ H \Tot (V) \to H \Tot (\K E\pi \tp V^{\ten 2}), \] and show that the map
\[ \algr: \K \Tot (X) \to \Tot (\K X) \] takes the usual `geometric' operations to these algebraic operations. This passage from geometric to algebraic operations is fairly straightforward and constitutes section~\ref{S:geomtoalg}. Notice that the main result from this section, Proposition~\ref{T:geomtoalg}, does not mention filtration quotients. We will deduce the behavior of $Q^m$ with respect to filtrations much later, in section~\ref{S:secondconv}.

More difficult is the passage from these algebraic operations to operations we have previously defined on the spectral sequence (see section~\ref{S:sso} and \cite{me1}). Part of the difficulty  is that we do not, in most cases, have formulas for the $\ext^m$ in the spectral sequence. Luckily, since we are working with cosimplicial simplicial $\K$-modules instead of just cosimplicial spaces, we can use the universal property of $\K \bkspace_{(\infty, s ,t)}$.

That said, we will not use these universal examples when showing convergence for $Q^{t-s}$. This bottom operation is the squaring operation by \cite[Proposition 10.5]{me1}, so convergence of the bottom operation is a special case of convergence of multiplication. This is the topic of section~\ref{S:tensorproducts}; our main multiplication theorems are Theorem~\ref{T:externalmult} and Corollary~\ref{C:mult}, while Theorem~\ref{T:bottomop} is specifically about the bottom operation and is an essential component in the proof of Theorem~\ref{T:convergence}.

We prove algebraic convergence for the higher external operations first in the case of the cosimplicial simplicial modules $\K \bkspace_{(\infty, s,t)}$. This implies the general case since $\K \bkspace_{(\infty, s, t)}$ is universal for elements in $E_{-s,t}^\infty$.
For any cosimplicial simplicial module $V$, we will see that both $N(\K E\pi \tp (\Tot V)^{\ten 2})$ and $TCN\left(\K E\pi \tp  V^{\ten 2}\right)$ are comodules over $S_*(B\pi)$ (see Propositions~\ref{P:nzh} and \ref{P:nnzh}). It follows that $H(\K E\pi \tp (\Tot V)^{\ten 2})$ and $H(\Tot (\K E\pi \tp V^{\ten 2}))\cong HTCN\left(\K E\pi \tp  V^{\ten 2}\right)$ are comodules over $H(B\pi)$. We will show (Theorem~\ref{T:comodules}) that the interchange map $\K E\pi \tp (\Tot V)^{\ten 2} \to \Tot (\K E\pi \tp V^{\ten 2})$ induces a map of $H(B\pi)$-comodules in homology. 

We use this fact in the special case of the cosimplicial simplicial module $\K \bkspace_{(\infty, s,t)}$. In this case, it turns out that $H\left(\K E\pi \tp \Tot (\K \bkspace_{(\infty, s ,t)})^{\ten 2}\right)$ is cofree over $H(B\pi)$, which allows us to show that the interchange map 
\[ H(\intalg): H\left(\K E\pi \tp \Tot (\K \bkspace_{(\infty, s ,t)})^{\ten 2}\right) \to H\left(\Tot \left(\K E\pi \tp  (\K \bkspace_{(\infty, s ,t)})^{\ten 2}\right)\right) \]
from \eqref{E:algint} is an injection. 
Other calculations tell us that these modules are isomorphic (Lemma~\ref{L:Hofhomotopyorbits}), but not that this map is an isomorphism. 
Everything is finite dimensional, so this establishes Proposition~\ref{P:interchange}, which states that the above injection is an isomorphism.

Finally, we have a complete description of the filtration quotients for
$H_{t-s} (\Tot \K \bkspace_{(\infty, s ,t)})$ and $H_{t-s+m} (\Tot (\K E\pi \tp (\K \bkspace_{(\infty, s ,t)})^{\ten 2}))$
(see Proposition~\ref{P:filtrquotients}). As remarked in section~\ref{S:sso}, all spectral sequence operations on $\imath \in E^\infty_{-s,t} (\bkspace_{(\infty, s, t)})$ are nontrivial, from which we deduce Theorem~\ref{T:convergence} in this special case. The general result then follows from universality and naturality.

\subsection{A note on the proof sketch \texorpdfstring{of 
\cite[5.11]{turner}}{in Turner
}}\label{S:turner}
We briefly mention several places where complications arise when one tries to fill out this proof sketch, which was attributed to McClure. 
%In order to be explicit, we use line numbers which refer to those on \cite[p.~3834]{turner}, while notation will come from the present paper.
It is implied that one may (up to homology) interchange the functors $\Tot$ and $X\mapsto E\pi \times_\pi X^{\times 2}$. It seems that one may need to restrict to cosimplicial spaces which are both cofibrant and fibrant (in either the Reedy or the projective model structures), a restriction which would exclude the Bousfield-Kan examples $\bkspace_{(\infty, s,t)}$. In that special case we prove an algebraic version in Proposition~\ref{P:interchange}.

Care should be taken in distinguishing the homotopy and homology spectral sequences; for instance, the spaces $\bkspace_{(\infty, s,t)}$ are not universal for elements in the homotopy spectral sequence as asserted (see \cite[\S 8.4]{bk} for a note on the complications). Rather, this is only true if one is working with cosimplicial simplicial groups and applies the free group functor to $\bkspace_{(\infty, s,t)}$ (or the free abelian group functor, if working in the abelian setting). See \cite[\S 6.1]{bk}. 

%It is stated on line 9 that the spectral sequence associated to $\bkspace_{(\infty, s,t)}$ converges, which does not seem to follow from the results of \cite{bousfield}. 
Finally, the author found (5.4) somewhat difficult to compute; the interested reader may find this calculation in \cite{me1}.

%%%%%%%%%%%%%%%%%%%%%%%%%%%%%%%%%%%%%%%%%%%%%%%%%%
%%%%%%%%%%%%%%%%%%%%%%%%%%%%%%%%%%%%%%%%%%%%%%%%%%
\section{Geometric to Algebraic}\label{S:geomtoalg}
%%%%%%%%%%%%%%%%%%%%%%%%%%%%%%%%%%%%%%%%%%%%%%%%%%
%%%%%%%%%%%%%%%%%%%%%%%%%%%%%%%%%%%%%%%%%%%%%%%%%%

Let $X$ be a cosimplicial space. We now pass from geometric operations on $\Tot X$ to algebraic operations on $\Tot \K X$.

We first describe two interchange maps. The first is a geometric interchange map
\begin{equation} \inta: E\pi \times_\pi (\Tot X)^{\times 2} \to \Tot (E\pi \times_\pi X^{\times 2}), \label{E:interchange} \end{equation} where $X$ is a cosimplicial space. Consider $(e,f,g)$ a $q$-simplex of $E\pi \times_\pi (\Tot X)^{\times 2}$, so that
\begin{align*}
e:& \Delta^q \to E\pi \\
f:& \Delta^\bullet \times \Delta^q \to X \\
g:& \Delta^\bullet \times \Delta^q \to X. \\
\end{align*}
We form $\inta(e,f,g)$ by using the iterated diagonal on $\Delta^\bullet \times \Delta^q$ and regarding $E\pi$ as a constant cosimplicial space. Formulaically it is given as
\begin{align*}
\inta(e,f,g): \Delta^\bullet \times \Delta^q &\to E\pi \times_\pi X^{\times 2} \\
(\faked_1, \faked_2) & \mapsto (e(\faked_2), f(\faked_1, \faked_2), g(\faked_1, \faked_2)).
\end{align*}
Similarly, if $V$ is a cosimplicial simplicial $\K$-module, we have an algebraic interchange map
\begin{equation} \intalg: \K E\pi \tp (\Tot V)^{\ten 2} \to \Tot (\K E\pi \tp V^{\ten 2}) \label{E:algint}\end{equation} which is given on an elementary tensor $e\ten f \ten g$ by
\[ (\faked_1, \faked_2) \mapsto e(\faked_2) \ten f(\faked_1, \faked_2) \ten g(\faked_1, \faked_2), \]
where $e$ is an element of $E\pi$.
%The passage from geometric to algebraic we are talking about is from operations on $\Tot X$ to operations on $\Tot \K X$, where $X$ is a cosimplicial space. 

\begin{lem}\label{L:geomalgint}
The following diagram commutes: 
\[ \xymatrix{ S_*(E\pi \times_\pi  (\Tot X)^{\times 2}) \ar@{->}[r] \ar@{->}[d]^{S_*\inta} & N (\K E\pi \tp (\Tot \K X)^{\ten 2}) \ar@{->}[d]^{N\intalg} \\
S_*(\Tot(E\pi \times_\pi X^{\times 2})) \ar@{->}[r] & N (\Tot (\K E\pi \tp (\K X)^{\ten 2})). }\]
\end{lem}
\begin{proof}
Both composites in the diagram
\[ \xymatrix{ E\pi \times_\pi  (\Tot X)^{\times 2} \ar@{->}[r] \ar@{->}[d]^\inta & \K E\pi \tp (\Tot \K X)^{\ten 2} \ar@{->}[d]^\intalg \\
\Tot(E\pi \times_\pi X^{\times 2}) \ar@{->}[r] & \Tot (\K E\pi \tp (\K X)^{\ten 2}) }\]
take $(e,f,g)$ in simplicial degree $q$ to the function with domain $\Delta^\bullet \times \Delta^q$ sending $(\faked_1, \faked_2)$ to $(e(\faked_2)) \ten (f(\faked_1, \faked_2)) \ten (g(\faked_1, \faked_2)).$
\end{proof}

\begin{prop}\label{T:geomtoalg} Let $X$ be a cosimplicial space.
For each $m$, the diagram
\[ \xymatrix{ 
H (\Tot X) \ar@{->}[r] \ar@{->}[d]^{Q^m} & H(\Tot \K X) \ar@{->}[d]^{Q^m_{alg}}\\
H (\Tot (E\pi \times_\pi X^{\times 2}) ) \ar@{->}[r] & H(\Tot (\K E\pi \tp (\K X)^{\ten 2})) \\
}\]
commutes, where $Q^m_{alg} = H(\intalg \nabla q^m)$ and the horizontal maps come from $\algr: \K \Tot \to \Tot \K$. 
\end{prop}
Notice that the horizontal maps induce injections on the filtration quotients $F^{-s} / F^{-s-1}$. The situation is necessarily more complicated for the vertical maps -- we will see that they often reduce the filtration degree.

\begin{proof}[Proof of proposition~\ref{T:geomtoalg}]
For each chain level operation $q^m$, consider the diagram
\[ \xymatrix{
S_* \Tot X \ar@{->}[r] \ar@{->}[d]^{q^m} & N \Tot \K X \ar@{->}[d]^{q^m}\\
W\tp ( S_* \Tot X)^{\ten 2}  \ar@{->}[r] \ar@{->}[d]^\nabla & W\tp (N \Tot \K X)^{\ten 2} \ar@{->}[d]^\nabla\\
S_* ( E\pi \times_\pi  (\Tot X)^{\times 2}) \ar@{->}[r]  & N (\K E\pi \tp (\Tot \K X)^{\ten 2})
}\]
where $S_* = N \K$ and the horizontal maps arise from the normalization of
$ \algr: \K \Tot X \to \Tot \K X$ from \eqref{E:algr}.
Naturality of $q^m$ and $\nabla$, along with the fact that $\K (E\pi \times_\pi X^{\times 2}) \cong \K E\pi \tp (\K X)^{\ten 2}$ imply that this diagram commutes. 
The result now follows from the preceding lemma.
%\[ (\faked_1, \faked_2) \mapsto (e\faked_2) \ten (f(\faked_1, \faked_2)) \ten (g(\faked_1, \faked_2)). \]
\end{proof}

%\begin{thm}\label{T:geomtoalg}
%The map $\K \Tot X \to \Tot \K X$ is compatible with external operations, in the sense that the diagram
%\[ \xymatrix{
%S_* \Tot X \ar@{->}[r] \ar@{->}[d]^{q^m} & N \Tot \K X \ar@{->}[d]^{q^m} \\
%S_*(\Tot(E\pi \times_\pi X^{\times 2})) \ar@{->}[r] & N (\Tot (\K E\pi \tp (\K X)^{\ten 2}))
%}\] commutes for every $m$.
%\end{thm}

%%%%%%%%%%%%%%%%%%%%%%%%%%%%%%%%%%%%%%%%%%%%%%%%%%
%%%%%%%%%%%%%%%%%%%%%%%%%%%%%%%%%%%%%%%%%%%%%%%%%%
\section{Tensor Product of Cosimplicial Simplicial Modules}\label{S:tensorproducts}
%%%%%%%%%%%%%%%%%%%%%%%%%%%%%%%%%%%%%%%%%%%%%%%%%%
%%%%%%%%%%%%%%%%%%%%%%%%%%%%%%%%%%%%%%%%%%%%%%%%%%

If $X$ and $X'$ are cosimplicial spaces, then $\Tot(X) \times \Tot (X') = \Tot(X\times X')$. When dealing with cosimplicial simplicial $\K$-modules $V$ and $V'$, it is more natural to ask how $\Tot$ behaves with respect to tensor products. The purpose of the first part of this section is to give the statement of Theorem~\ref{T:tensorproducts}, which implies that in some cases the map $\Tot(V) \ten \Tot(V') \to \Tot(V\ten V')$ is an isomorphism \emph{in homology}.

Let $U$ and $V$ be cosimplicial simplicial $\K$-modules. We now show that the diagonal gives 
 an interchange map \begin{equation} \intb: \Tot_\mess U \ten \Tot_\mess V \to \Tot_\mess (U\ten V) \label{E:intb} \end{equation} which is analogous to that of $\intalg$.
To be precise, suppose that we have an elementary tensor $(f)\ten (g)$ in the set of $q$-simplices on the left hand side. We have that $(f)$ and $(g)$ are cosimplicial maps
\begin{align*}
f: \sk_\mess \Delta^\bullet \times \Delta^q &\to U \\
g: \sk_\mess \Delta^\bullet \times \Delta^q &\to V,  
\end{align*}
which we regard as linear maps from
$\K \sk_\mess \Delta^\bullet \ten \K \Delta^q$. The diagonal on $\sk_\mess \Delta^\bullet \times \Delta^q$ then gives the composite 
\[ \K \sk_\mess \Delta^\bullet \ten \K \Delta^q \to (\K \sk_\mess \Delta^\bullet \ten \K \Delta^q) \ten (\K \sk_\mess  \Delta^\bullet \ten \K \Delta^q) \overset{f\ten g}{\longrightarrow} U \ten V \]
which is the image of $(f)\ten (g)$ under $\intb$.

We thus have a map
\[  N \Tot_\mess U \ten N \Tot_\mess V \overset{\nabla}\to N ( \Tot_\mess U  \ten  \Tot_\mess V) \overset{N\intb}\longrightarrow N \Tot_\mess(U\ten V) \] which produces
\[ H(\Tot_\mess U) \ten H(\Tot_\mess V) \to H (\Tot_\mess (U\ten V)). \]
The definition of $\intb$ mirrors, for cosimplicial spaces $X$ and $X'$, the usual isomorphism $\Tot_\mess X \times \Tot_\mess X' \overset\cong\to \Tot_\mess (X\times X')$. In fact, using the natural transformation $\algr: \K \Tot_\mess \to \Tot_\mess \K$, we have the following:
\begin{prop}\label{P:geomtoalgprod} If $X$ and $X'$ are cosimplicial spaces, then
the diagram
\[ \xymatrix{
H (\Tot_\mess X) \ten H(\Tot_\mess X') \ar@{->}[r] \ar@{->}[d]^\cong & H (\Tot_\mess \K X) \ten H(\Tot_\mess \K X') \ar@{->}[d]^{H(\intb \nabla)} \\
H(\Tot_\mess (X\times X')) \ar@{->}[r] & H(\Tot_\mess(\K X \ten \K X'))
}\]
commutes. \qed
\end{prop}

We next consider the composite
\begin{align*} T_\mess CN U \ten T_\mess CN V 
& \to T_\mess (CNU \ten CNV) 
\\ &\overset{AW}{\to} T_\mess C(NU \ten NV) 
\\ &\overset{\nabla}{\to} T_\mess CN(U \ten V), \end{align*}
which on homology induces 
$ H (T_\mess CNU) \ten H (T_\mess CNV) \to H(T_\mess CN (U\ten V)). $

\begin{thm}\label{T:tensorproducts} In homology, the map $\phi_\mess$ from \eqref{E:phiell} is compatible with tensor products, in the sense that the diagram
\[ \xymatrix{
H(\Tot_\mess U) \ten H(\Tot_\mess V) \ar@{->}[r]^-{\phi_\mess \ten \phi_\mess} \ar@{->}[d]^{H(\intb \nabla)}  &  H (T_\mess CNU) \ten H (T_\mess CNV) \ar@{->}[d]^{ H(\nabla AW)} \\
H (\Tot_\mess (U\ten V)) \ar@{->}[r]^{\phi_\mess} & H(T_\mess CN(U\ten V))
}\] commutes.
\end{thm}
Section~\ref{S:Ptensor} is devoted to the proof of this theorem.

% \begin{rem}
% The tensor product is the wrong target for the external squaring operation, since, for example, a noncommutative square is not a homomorphism. The correct target should be the homotopy orbits. 
% \end{rem}

We wish the external squaring operation to be a homomorphism, so we must use the homotopy orbits of the tensor product rather than the tensor product itself.
For a chain complex $C$, the standard passage from $C\ten C$ to the homotopy orbits $W\tp (C\ten C)$ is given by $c\ten c' \mapsto e_0 \ten c \ten c'$. We now extend this to the simplicial setting.

\begin{defn}\label{D:xi}
Let $Z$ be a simplicial $\K$-module. Define a map of simplicial $\K$-modules $\xi$ by \begin{align*} \xi: Z \ten Z &\to \K E\pi \tp (Z\ten Z) \\
z \ten z' &\mapsto * \ten z \ten z'. \end{align*} 
By `$*\in E\pi$', we mean the image of $e$ under the inclusion of the fiber $\pi \to E\pi$; if we use a standard bar construction for $E\pi$\label{Label:standardbarconstruction} (see \cite[p.~257 and p.~269--270]{weibel}, \cite[V.4]{gj}, or \cite[p.~87--88]{mayso}), we may write
\[ *=(\underbrace{e, \cdots , e}_{q+1}) \text{ if } z, z' \in Z_q. \] 
\end{defn}

The reason for this definition is the following lemma:

\begin{lem} Let $Z$ be a simplicial $\K$-module. Then the diagram
\[ \xymatrix{
& N(Z\ten Z) \ar@{->}[dr]^{N\xi} \ar@{->}[dl]_{e_0 \ten -} & \\
W\tp N(Z\ten Z) \ar@{->}[rr]^{\nabla} && N(\K E\pi \tp (Z\ten Z))
}\]
commutes. \qed
\end{lem}

%, and notice that it would be desirable to have a map $\xi$ of simplicial $\K$-modules making the diagram

%commute. One checks that it commutes if we define $\xi$ to be 

%where $*$ is induced from the identity element of $\pi$ in the fibration $\pi \to E\pi \to B\pi$. 
%using the standard `$W$' or bar construction as our model for $E\pi$, we have \[ *=(\underbrace{e, \cdots , e}_{q+1}) \text{ if } z, z' \in Z_q. \] 

\begin{prop}\label{P:tensortocommtensor} If $V$ is a cosimplicial simplicial $\K$-module, then the diagram
\[ \xymatrix{
H_n(\Tot_\mess V) \ar@{->}[r]^{\phi_\mess} \ar@{->}[d]^{Q^n} & H_nT_\mess CN V \ar@{->}[dd]^{Q^n} \\
H_{2n}(\K E\pi \tp (\Tot_\mess V)^{\ten 2}) \ar@{->}[d]^{H(\intalg)} & \\
H_{2n}(\Tot_\mess (\K E\pi \tp V^{\ten 2})) \ar@{->}[r]^{\phi_\mess} & H_{2n}T_\mess CN(\K E\pi \tp V^{\ten 2})
}\]
commutes.
\end{prop}
\begin{proof} If $Z$ is a simplicial module, then the bottom external operation is defined at the chain level by the composite of maps in the commutative diagram
\[ \xymatrix{ N(Z) \ar@{->}[r]^-{\text{diag}} &
N(Z) \ten N(Z) \ar@{->}[d]^\nabla \ar@{->}[r]^-{e_0\ten -} & W\tp N(Z) \ten N(Z) \ar@{->}[d]^{1\ten \nabla}_\simeq  \\
&N(Z \ten Z) \ar@{->}[r]^-{e_0\ten -} \ar@{->}[dr]_\xi & W\tp N(Z\ten Z) \ar@{->}[d]^\nabla_\simeq \\
&& N(\K E\pi \tp Z \ten Z) }\]
which becomes a homomorphism in homology. 
By \[ Q^n: H_n(\Tot_\mess V) \to H_{2n}(\K E\pi \tp (\Tot_\mess V)^{\ten 2})\] we mean this operation with $Z=\Tot_\mess V$.

Let $V$ be a cosimplicial simplicial module and let $f,g \in (\Tot_\mess V)_q$. Considering $f$ and $g$ as maps $\sk_\mess \Delta^\bullet \times \Delta^q \to V$, we can evaluate that both composites in the diagram
\[ \xymatrix{
\Tot_\mess V \ten \Tot_\mess V \ar@{->}[r]^{\intb} \ar@{->}[d]^\xi & \Tot_\mess (V\ten V) \ar@{->}[d]^{\Tot_\mess(\xi)} \\
\K E\pi \tp (\Tot_\mess V)^{\ten 2} \ar@{->}[r]^{\intalg} & \Tot_\mess(\K E\pi \tp V^{\ten 2})
}\]
take $f\ten g$ to the map
\begin{align*}
\sk_\mess \Delta^\bullet \times \Delta^q & \to \K E\pi \tp V^{\ten 2} \\
a & \mapsto * \ten f(a) \ten g(a),
\end{align*}
so this diagram commutes. Since $\phi_\mess$ is a natural transformation, the diagram
\[ \xymatrix{
H \Tot_\mess (V\ten V) \ar@{->}[r]^{\phi_\mess} \ar@{->}[d]^{H\Tot_\mess (\xi)} & HT_\mess CN(V\ten V) \ar@{->}[d]^{H\xi} \\
H\Tot_\mess (\K E\pi \tp V^{\ten 2}) \ar@{->}[r]^{\phi_\mess} & HT_\mess CN( \K E\pi \tp V^{\ten 2})
}\]
commutes. The result follows by combining this diagram with the homology of the previous one.
\end{proof}

%%%%%%%%%%%%%%%%%%%%%%%%%%%%%%%%%%%%%%%%%%%%%%%%%%
%%%%%%%%%%%%%%%%%%%%%%%%%%%%%%%%%%%%%%%%%%%%%%%%%%
\subsection{Bottom Operation}
%%%%%%%%%%%%%%%%%%%%%%%%%%%%%%%%%%%%%%%%%%%%%%%%%%
%%%%%%%%%%%%%%%%%%%%%%%%%%%%%%%%%%%%%%%%%%%%%%%%%%

We now show that Theorem~\ref{T:convergence} holds in the special case when $m=t-s$.
We must examine the way the operation behaves with respect to filtrations. To simplify notation, we begin with the bicomplex case. 

Let $B$ be a second-quadrant bicomplex. The filtration on the homology of the totalization of $B$ is given by \[ F^{-s} H(TB) = \ker (H TB\to H T_{s-1} B), \] where $T_{s-1} B = TB / F^{-s} TB$ as in \eqref{E:totalell}.

\begin{lem} Let $B$ and $B'$ be second-quadrant bicomplexes. Then the map 
\[ HTB \ten HTB' \to H(TB \ten TB') \to HT(B\ten B') \] 
takes $F^{-s} \ten F^{-s'}$ into $F^{-s-s'}$.
\end{lem}
\begin{proof}
We will let $B$ and $B'$ be (second-quadrant) bicomplexes, and
\begin{align*} [x] &\in F^{-s} H(TB) = \ker(H TB \to H T_{s-1} B) \\
[y] & \in F^{-s'} H(TB'). \end{align*}
One consequence is that there is an element $a$ so that $x+\partial a \in F^{-s} TB$, thus we can and will assume that the representatives are already in the appropriate filtration:
\begin{align*} x &\in F^{-s} TB \\ y&\in F^{-s'} TB'. \end{align*}
We then have $x\ten y \in F^{-s-s'} T(B\ten B')$.
%Then \[ x\ten y \in F^{-s-s'} T(B\ten B'), \]
%thus the map
%\[ HTB \ten HTB' \to HT(B\ten B') \] takes $F^{-s} \ten F^{-s'}$ into $F^{-s-s'}$.
\end{proof}

We now return to the setting where we have a cosimplicial simplicial $\K$-module $V$, and prove Theorem~\ref{T:convergence} for for the special case when $m=t-s$. The map $\abut$ in the theorem statement is the abutment map from \eqref{SYMBOL:ABUTMENT}.

\begin{thm}\label{T:bottomop}
Suppose that $X$ is a cosimplicial space. Then
\[ Q^{t-s} [ F^{-s} H_{t-s} (\Tot (X))] \ci F^{-2s}H_{2(t-s)} (\Tot(E\pi \times_\pi X^{\times 2})) \]
and the following diagram commutes:
\[ \xymatrix{
\left({F^{-s}}/{F^{-s-1}}\right) H_{t-s} (\Tot X) \ar@{->}[r]^-\abut \ar@{->}[dd]^{Q^{t-s}} & E_{-s,t}^\infty (S_*(X)) \ar@{->}[d]^{\ext^{t-s}} \\
& E_{-2s, 2t}^\infty (W\tp S_*(X)^{\ten 2}) \ar@{->}[d]^{\cong}_{E^1 - iso} \\
\left( {F^{-2s}}/{F^{-2s-1}}\right) H_{2(t-s)} (\Tot (E\pi \times_\pi X^{\times 2})) \ar@{->}[r]^-\abut & E_{-2s, 2t}^\infty (S_*(E\pi \times_\pi X^{\times 2})).
}\]
\end{thm}
\begin{proof} In light of Lemma~\ref{L:geomalgint}, we consider $V=S_*X$.
 The filtration on $H\Tot V$ is given by
\[ F^{-s} H \Tot V = \ker (H\Tot V \overset{\phi}{\to} HTCN V \to HT_{s-1} CN V). \]
The $\mess = \infty$ case of 
Theorem~\ref{T:tensorproducts} asserts the commutativity of the following diagram:
\[ \xymatrix@C+6pt{
H\Tot V \ten H\Tot V \ar@{->}[r]^-{\phi_\infty \ten \phi_\infty} \ar@{->}[dd]^{\intb}   & HTCNV \ten HTCN V \ar@{->}[d] \\
 & HT[CNV \ten CN V] \ar@{->}[d]^{\nabla AW}  \\
H\Tot (V\ten V) \ar@{->}[r]^{\phi_\infty} & HTCN(V\ten V).
}\]
Since the diagram
\[ \xymatrix{
& HT[CNV \ten CN V] \ar@{->}[d]^{\nabla AW} \ar@{->}[r] & HT_{s+s'-1}[CNV \ten CN V] \ar@{->}[d]^{\nabla AW} \\
H\Tot (V\ten V) \ar@{->}[r]^{\phi_\infty} & HTCN(V\ten V) \ar@{->}[r]& HT_{s+s'-1}CN(V\ten V)
}\]
also commutes, we see that we have an induced map
\[  ( F^{-s} / F^{-s-1}) H\Tot V\ten ( F^{-s'} / F^{-s'-1}) H\Tot V  \to (F^{-s-s'} / F^{-s-s'-1}) H\Tot (V\ten V)  \] and an induced diagram
%FORMATTING
\[ \scalebox{.75}{ 
	\xymatrix{
	(F^{-s} / F^{-s-1}) H\Tot V \ten (F^{-s'} / F^{-s'-1}) H\Tot V \ar@{->}[r] \ar@{->}[d]   & (F^{-s} / F^{-s-1}) HTCNV \ten  (F^{-s'} / F^{-s'-1}) HTCN V \ar@{->}[d] \\
	(F^{-s-s'}/F^{-s-s'-1}) H\Tot (V\ten V) \ar@{->}[r] & (F^{-s-s'}/F^{-s-s'-1}) HTCN(V\ten V).
	} 
}
\]
Setting $s=s'$ and combining this with Propositions~\ref{T:geomtoalg} and~\ref{P:tensortocommtensor} completes the proof.
\end{proof}

The more general statement is that the abutment is compatible with external multiplication:
\begin{thm}\label{T:externalmult}
The following diagram commutes.
%FORMATTING
\[ \xymatrix%@C-10pt
{
(F^{-s} / F^{-s-1}) H\Tot X \ten (F^{-s'} / F^{-s'-1}) H\Tot X \ar@{->}[r]^-{\abut \ten \abut} \ar@{->}[d]   & E^\infty_{-s} (S_*(X)) \ten  E^\infty_{-s'} (S_*(X)) \ar@{->}[d] \\
(F^{-s-s'}/F^{-s-s'-1}) H\Tot (X\times X) \ar@{->}[r]^-\abut \ar@{->}[d] & E^\infty_{-s-s'} (S_*(X\times X)) \ar@{->}[d]\\
(F^{-s-s'}/F^{-s-s'-1}) H\Tot (E\pi \times_\pi X^{\times 2}) \ar@{->}[r]^-\abut & E^\infty_{-s-s'} (S_*(E\pi \times_\pi X^{\times 2})) \\
}\]
\end{thm}
\begin{proof}
The bottom square commutes by naturality. Using the $\mess = \infty$ case of Proposition~\ref{P:geomtoalgprod} and the definition of the filtrations, we see that the square
\[ \scalebox{.75}{ 
	\xymatrix{
	(F^{-s} / F^{-s-1}) H\Tot X \ten (F^{-s'} / F^{-s'-1}) H\Tot X \ar@{->}[r] \ar@{->}[d]   
	& (F^{-s} / F^{-s-1}) H \Tot \K X \ten  (F^{-s'} / F^{-s'-1}) H\Tot  \K X \ar@{->}[d] \\
	(F^{-s-s'}/F^{-s-s'-1}) H\Tot (X\times X) \ar@{->}[r] 
	& (F^{-s-s'}/F^{-s-s'-1}) H\Tot (\K X \ten \K X)
	} 
}
\]
commutes. Combining this with the last commutative diagram from the proof of 
Theorem~\ref{T:bottomop} and the fact that $\K X \ten \K X = \K (X\times X)$, we see that the top square commutes.
\end{proof}

When $X$ is a cosimplicial $\msc$-space, we 
can use the previous theorem and apply the structure map
\[ E\pi \times_\pi X^{\times 2} \to X \] to relate the two internal multiplications.
%obtained by following the above diagram with the structure map $E\pi \times_\pi X^{\times 2} \to X$.

\begin{cor}\label{C:mult}If $X$ is a cosimplicial $\msc$-space, then the multiplication $\mu$ on $H\Tot X$ is compatible with the filtration
in the sense that
\[ \mu(F^{-s} \ten F^{-s'}) \ci F^{-s-s'}. \] Furthermore, $\mu$ descends to the quotient and the diagram
%FORMATTING
\[ \xymatrix%@C-10pt
{
(F^{-s} / F^{-s-1}) H\Tot X \ten (F^{-s'} / F^{-s'-1}) H\Tot X \ar@{->}[r]^-{\abut \ten \abut} \ar@{->}[d]^\mu   & E^\infty_{-s} (S_*(X)) \ten  E^\infty_{-s'} (S_*(X)) \ar@{->}[d]^\mu \\
(F^{-s-s'}/F^{-s-s'-1}) H\Tot X \ar@{->}[r]^-\abut & E^\infty_{-s-s'} (S_*(X)) \\
}\]
commutes.  \qed
\end{cor}

%%%%%%%%%%%%%%%%%%%%%%%%%%%%%%%%%%%%%%%%%%%%%%%%%%
%%%%%%%%%%%%%%%%%%%%%%%%%%%%%%%%%%%%%%%%%%%%%%%%%%
\section{Comodules over \texorpdfstring{$H(B\pi)$}{H(B\unichar{960})}}\label{S:comodules}
%%%%%%%%%%%%%%%%%%%%%%%%%%%%%%%%%%%%%%%%%%%%%%%%%%
%%%%%%%%%%%%%%%%%%%%%%%%%%%%%%%%%%%%%%%%%%%%%%%%%%

Let $V$ be a cosimplicial simplicial $\K$-module. 
In this section we show that both
\[ N(\K E\pi \tp (\Tot(V))^{\ten 2}) \text { and } TCN(\K E\pi \tp V^{\ten 2}) \] 
are comodules over the coalgebra $H(B\pi)$. Using the fact that $\phi_\infty$ is a quasi-isomorphism, this gives that $H(\Tot(\K E\pi \tp V^{\ten 2}))$ is also a comodule over $H(B\pi)$. We also will state Theorem~\ref{T:comodules}, which implies that the homology of the interchange map
\[ H(\intalg): H(\K E\pi \tp (\Tot(V))^{\ten 2}) \to H(\Tot(\K E\pi \tp V^{\ten 2})) \] is a map of $H(B\pi)$-comodules. This fact will be important in the next section when we specialize to a case where the left hand side is cofree over $H(B\pi)$.

We will write $\psi$ for coproducts, for example
 \[ \psi_{E\pi}: \K E\pi \to \K [ E\pi \times E\pi ] \cong \K E\pi \ten \K E\pi\] is the map induced from the diagonal of $E\pi$.
 
\begin{lem}\label{L:rhoone}
Let $Z$ be a simplicial $\K \pi$-module.
 The homotopy orbits
\[ \zh := \K E\pi \tp Z \]
has a natural coassociative coaction by $\K B\pi$ induced from the diagonal on $E\pi$. \end{lem}

\begin{proof} The coaction is defined by
\begin{align*} \coact_1: \zh = \K E\pi \tp Z &\cong \K E\pi \tp (\K * \ten Z) \\ &\to (\K E\pi \ten \K E\pi) \tp (\K * \ten Z) \\ 
& \to (\K E\pi \ten \K E\pi) \ten_{\pi \times \pi} (\K * \ten Z) \\
&\cong (\K E\pi \tp \K *) \ten (\K E\pi \tp Z)  = \K B\pi \ten \zh. \end{align*}
The comultiplication on $\K B\pi$ is also induced from $\psi_{E\pi}$, which shows that $\coact_1$ is coassociative:
\begin{equation}\label{E:Zhpi} \xymatrix{
\zh \ar@{->}[r]^{\coact_1} \ar@{->}[d]^{\coact_1}& \K B\pi \ten \zh \ar@{->}[d]^{1\ten \coact_1} \\
\K B\pi \ten \zh \ar@{->}[r]^-{ \psi_{B\pi} \ten 1} &  \K B\pi \ten \K B\pi \ten \zh.
}\end{equation}

\end{proof}

\begin{rem} Suppose $a\in E\pi_q$ and $z\in Z_q$. Then 
the formula for $\coact_1$ is
\[ \coact_1(a\ten z) = \bar{a} \ten (a\ten z). \]
\end{rem}

For simplicity, we use the model for $E\pi$ from section~\ref{S:notationbackground}, so that $N\K E\pi = W$. 
Notice that the map
\[ \psi_W : W=N\K E\pi \overset{N\psi_{E\pi}}{\longrightarrow} N(\K E\pi \ten \K E\pi) \overset{AW}{\longrightarrow} W\ten W \] is coassociative (by associativity of the Alexander-Whitney map) and sends $e_m$ to 
$\sum_{i+j=m} e_i \ten \sigma^i e_j. $
Write \[ \bw = W\tp \K = N\K B\pi = H (B\pi), \] which has comultiplication induced from $\psi_W$. Namely, $\psi_{\bw}(\bar e_m) = \sum \bar e_i \ten \bar e_j$.

\begin{prop}\label{P:nzh} Let $Z$ be a simplicial $\K$-module. Then $N\zh$ is naturally a comodule over $\bw$.
\end{prop}
\begin{proof}
We write $\coact_2$ for the composite
\begin{equation} N \zh \overset{N\coact_1}{\longrightarrow} N( \K B\pi \ten \zh ) \overset{AW}{\longrightarrow} \bw \ten N\zh. \label{E:act2} \end{equation}
Associativity of the Alexander-Whitney map and the commutativity of (\ref{E:Zhpi}) gives commutativity of the following diagram.
\[  \xymatrix{
N\zh \ar@{->}[r]^{\coact_2} \ar@{->}[d]^{\coact_2}& \bw \ten N\zh \ar@{->}[d]^{1\ten \coact_2} \\
\bw \ten N\zh \ar@{->}[r]^-{ \psi_{\bw} \ten 1} &  \bw \ten \bw \ten N\zh
}\]
Thus the coaction is coassociative.
\end{proof}

We now consider a cosimplicial simplicial $\K$-module $V$ and the interchange map
\[ \intalg: \K E\pi \tp (\Tot (V))^{\ten 2} \to \Tot( \K E\pi \tp (V^{\ten 2})) \] from \eqref{E:algint}. Normalizing, we have
\[ N(\K E\pi \tp (\Tot(V))^{\ten 2}) \overset{N(\intalg)}{\longrightarrow} N \Tot (\K E\pi \tp V^{\ten 2})) \overset{\simeq}{\to} TCN  (\K E\pi \tp V^{\ten 2})), \]
where the quasi-isomorphism on the right is the map $\phi_\infty$.
Proposition~\ref{P:nzh} tells us that $N(\K E\pi \tp (\Tot(V))^{\ten 2})$ is a comodule over $\bw$. 
We also have
\begin{prop}\label{P:nnzh}
Let $V$ be a cosimplicial simplicial $\K$-module. Then \[TCN (\K E\pi \tp V^{\ten 2})\] is naturally a comodule over $\bw$.
\end{prop}
\begin{proof}
We let $Z=V^{\ten 2}$ and consider the coaction on the cosimplicial chain complex $N\zh$ by applying $\coact_2$ from Proposition~\ref{P:nzh} in each cosimplicial degree. We have the following diagram,
\[ \xymatrix{ 
CN\zh \ar@{->}[r]^{C(\coact_2)} \ar@{->}[d]^{C(\coact_2)} 
& C(\bw \ten N\zh) \ar@{->}[r]^\nabla \ar@{->}[d]_{C(1\ten \coact_2)} 
& C(\bw) \ten CN\zh \ar@{->}[d]_{1\ten C(\coact_2)} 
\\C(\bw \ten N\zh) \ar@{->}[r]^-{C(\psi_{\bw} \ten 1)} \ar@{->}[d]^\nabla 
& C(\bw \ten \bw \ten N\zh) \ar@{->}[r]^\nabla \ar@{->}[d]^\nabla 
& C(\bw) \ten C(\bw \ten N\zh)  \ar@{->}[d]^\nabla 
\\C(\bw) \ten CN\zh \ar@{->}[r]^-{C(\psi_{\bw})\ten 1} & C(\bw \ten \bw) \ten CN\zh \ar@{->}[r]^\nabla & C(\bw) \ten C(\bw) \ten CN\zh
}\]
which commutes by Proposition~\ref{P:nzh} (upper-left), associativity of the shuffle map $\nabla$ (bottom-right), and the fact that $\nabla$ is a natural transformation. Here we are considering $\bw$ as a constant cosimplicial complex. Furthermore, $C(\bw) = \bw^v$ where $\bw^v$ is the bicomplex which is zero outside of bidegrees $(0,*)$ and $\bw^v_{0,q} = \bw_q$. Finally, we have that $T(\bw^v \ten B) \cong \bw \ten TB$  when  $B$  is a second-quadrant bicomplex using the natural transformation $T(-) \ten T(-) \to T(-\ten -)$. We define the coaction 
\begin{equation} \coact_3:  TCN\zh \to \bw \ten TCN\zh \label{E:act3}\end{equation} as the map induced from $\nabla \coact_2$, which we see is coassociative by the above commutative diagram and naturality of the isomorphism $T(\bw^v \ten B) \cong \bw \ten TB$. 
%
%The fact that $TCN  (\K E\pi \tp V^{\ten 2}))$ is a $\bw$-comodule comes from the diagram
%\[ \xymatrix{
%TCN\zh \ar@{->}[r]^{\coact_2} \ar@{->}[d]^{\coact_2}& TC (\bw \ten N\zh) \ar@{->}[d]^{1\ten \coact_2} \\
%TC(\bw \ten N\zh) \ar@{->}[r]^-{ \psi_{\bw} \ten 1} &  TC(\bw \ten \bw \ten N\zh
%)}\]
%for $Z=V^{\ten 2}$, the fact that $
%C(\bw) = \bw^v$ if $\bw$ is considered as a constant cosimplicial complex, and the fact that $T(\bw^v \ten B) \cong \bw \ten TB$  when  $B$  is a second-quadrant bicomplex. Here $\bw^v$ is the bicomplex which is zero outside of bidegrees $(0,*)$ and $\bw^v_{0,q} = \bw_q$.
\end{proof}

\begin{rem}
Let $Z$ be a cosimplicial simplicial $\K \pi$-module. It is unclear how to make $\Tot \zh$ into a $\K B\pi$-comodule, although we can easily make $H(\Tot \zh)$ into an $H(B\pi)$-comodule by declaring that  $H\phi_\infty$ is an isomorphism of homology comodules. The coaction is given by the diagram
\[ \xymatrix{
N\Tot \zh \ar@{->}[r]^-{N\Tot \coact_1} \ar@{->}[d]_{\phi_\infty}^\simeq& N\Tot(\K B\pi \ten \zh) \ar@{->}[d]_{\phi_\infty}^\simeq \\
TCN \zh \ar@{->}[r]^-{TCN \coact_1} & TCN(\K B\pi \ten \zh) \ar@{->}[d] \\
& \bw \ten TCN\zh \\
& \bw \ten N\Tot \zh \ar@{->}[u]_{1\ten \phi_\infty}^\simeq
}\]
\end{rem}

The key result in the proof of Proposition~\ref{P:interchange} is the following theorem, whose proof will be deferred until Section~\ref{S:prooftcomodules}.

\begin{thm}\label{T:comodules} Let $V$ be a cosimplicial simplicial $\K$-module. Then \[
\phi_\infty N(\intalg): N(\K E\pi \tp (\Tot(V))^{\ten 2}) \to  TCN  (\K E\pi \tp V^{\ten 2}) \] is a map of $\bw$-comodules.
\end{thm}

%%%%%%%%%%%%%%%%%%%%%%%%%%%%%%%%%%%%%%%%%%%%%%%%%%
%%%%%%%%%%%%%%%%%%%%%%%%%%%%%%%%%%%%%%%%%%%%%%%%%%
\section{Completion of the Proof of Theorem~\ref{T:convergence}}\label{S:secondconv}
%%%%%%%%%%%%%%%%%%%%%%%%%%%%%%%%%%%%%%%%%%%%%%%%%%
%%%%%%%%%%%%%%%%%%%%%%%%%%%%%%%%%%%%%%%%%%%%%%%%%%
We now have all of the building blocks in place to prove our main theorem:

\theoremstyle{plain}
\newtheorem*{mainThm}{Theorem \ref{T:convergence}}

\begin{mainThm} Suppose that $X$ is a cosimplicial space. Then
\[ Q^m [ F^{-s} H_{t-s} (\Tot (X))] \ci F^{-v}H_{t-s+m} (\Tot(E\pi \times_\pi X^{\times 2})) \]
where \[ v(t,s,m) = \begin{cases} s & m\geq t \\ t+s-m & t-s \leq m \leq t. \end{cases} \]Furthermore, for each $m$  the following diagram commutes:
\[ \xymatrix{
\left({F^{-s}}/{F^{-s-1}}\right) H_{t-s} (\Tot X) \ar@{->}[r]^-\abut \ar@{->}[dd]^{Q^m} & E_{-s,t}^\infty (S_*(X)) \ar@{->}[d]^{\ext^m} \\
& E_{-v, v-s+m+t}^\infty (W\tp S_*(X)^{\ten 2}) \ar@{->}[d]^{\cong}_{E^1 - iso} \\
\left( {F^{-v}}/{F^{-v-1}}\right) H_{t-s+m} (\Tot (E\pi \times_\pi X^{\times 2})) \ar@{->}[r]^-\abut & E_{-v, v-s+m+t}^\infty (S_*(E\pi \times_\pi X^{\times 2})).
}\]
%where the horizontal maps are those given by Bousfield.
\end{mainThm}

By Theorem~\ref{T:geomtoalg} we need only to prove the algebraic version. Namely, for a cosimplicial simplicial module $V$, we must show that the external operations  $Q^m$ on $H(\Tot V)$ land in the correct filtration degree and that the diagram
\[ \xymatrix{
(F^{-s} / F^{-s-1}) H_{t-s} (\Tot V) \ar@{->}[r] \ar@{->}[dd]^{Q^m} & E_{-s,t}^\infty (NV) \ar@{->}[d]^{\ext^m} \\
& E_{-v, v-s+m+t}^\infty (W\tp NV^{\ten 2}) \ar@{->}[d]^{\cong}_{E^1 - iso} \\
(F^{-v} / F^{-v-1}) H_{t-s+m} (\Tot (\K E\pi \tp V^{\ten 2})) \ar@{->}[r] & E_{-v, v-s+m+t}^\infty (N(\K E\pi \tp V^{\ten 2}))
}\]
commutes. But notice that each element of
\[ (F^{-s}/F^{-s-1}) H_{t-s} \Tot V \hookrightarrow E_{-s,t}^\infty (V) \]
is represented by a map $\K \bkspace_{(\infty, s, t)} \to V$, so by universality (see section~\ref{S:universality}) we may prove convergence for the special case of $\K \bkspace_{(\infty, s, t)}$. 

For the rest of the section we restrict to this special case and abbreviate \[ U=\K \bkspace_{(\infty, s, t)}.\] 

\begin{lem}\label{L:totV} Let $U=\K \bkspace_{(\infty, s, t)}$.
Then
\[ H_k (\Tot U) = \begin{cases} \K & k=t-s \\  0 & k\neq t-s. \end{cases}\]
Furthermore, the filtration on $H(\Tot U)$ is given by \[ 
\K = F^{-s} = F^{-s+1} = \cdots = H(\Tot U) \]  
and $F^{-s-1} = 0$.
\end{lem}
\begin{proof}
Note that, using the notation of \cite[Section 2.2]{me1}, $NU = D_{\infty s t}$. 
By \cite[Lemma 2.2]{bousfield}, we have $H(\Tot U) = H(TCNU) = H(TC D_{\infty st})$, so the first statement is a consequence of \cite[Proposition 2.4]{me1}. In the proof of that proposition we explicitly wrote down the bicomplex $C(D_{\infty st})$, which we reproduce as figure~\ref{F:inftyconormed}. 
\begin{figure}[ht]
\centering
\begin{tikzpicture}
    %Axes
    \draw[->] (0.2,0) -- (-4,0);
    \draw[->] (0,-0.2) -- (0,4);
    %Horizontal ticks
    \draw (-.5,2pt) -- (-.5,-2pt) node[below] {$-s$};
    \draw (-1,2pt) -- (-1,-2pt); % node[below] {$-(s+r)$};
%    \draw (-2.5,2pt) -- (-2.5,-2pt); 
%    \draw (-3,2pt) -- (-3,-2pt); %node[below] {$-v$}; 
    %Vertical Ticks
    \draw (2pt,.5) -- (-2pt,.5) node[right] {$t$};
    \draw (2pt,1) -- (-2pt,1) node[right] {$t+1$};
%    \draw (2pt,2.5) -- (-2pt,2.5) node[right] {$t+r-2$};
%    \draw (2pt,3) -- (-2pt,3) node[right] {$t+r-1$};
    %Points on x=-y
    \fill (-.5,.5) circle (2pt);
    \fill (-1,1) circle (2pt);
    \fill (-2.5,2.5) circle (2pt);
    \fill (-3,3) circle (2pt);
    %Points on x=-y-1
    \fill (-1,.5) circle (2pt);
    \fill (-1.5, 1) circle (2pt);
    \fill (-2.5,2) circle (2pt);
    \fill (-3,2.5) circle (2pt);
    \fill (-3.5,3) circle (2pt);
    %Horizontal Arrows
    \draw [->] (-.6,.5) -- (-.9,.5);
    \draw [->] (-1.1,1) -- (-1.4,1);
    \draw [->] (-2.1,2) -- (-2.4,2);
    \draw [->] (-2.6,2.5) -- (-2.9,2.5);
    \draw [->] (-3.1,3) -- (-3.4,3);    
    %Vertical Arrows
    \draw [->] (-1,.9) -- (-1,.6);
    \draw [->] (-1.5,1.4) -- (-1.5,1.1);
    \draw [->] (-2.5,2.4) -- (-2.5,2.1);
    \draw [->] (-3,2.9) -- (-3,2.6);
    \draw [->] (-3.5,3.4) -- (-3.5,3.1);
    %Dotted Lines
    \draw [dotted] (-1.5,1.5)--(-2,2);
%    \draw [dotted] (-1.5,1)--(-2.5,2);
    \draw [dotted] (-1.8,1.3)--(-2.2,1.7);
        \draw [dotted] (-3.5,3.5)--(-4,4);
         \draw [dotted] (-3.8,3.3)--(-4.2,3.7);
\end{tikzpicture}
\caption{The Bicomplex $C(D_{\infty st})$}\label{F:inftyconormed}
\end{figure}
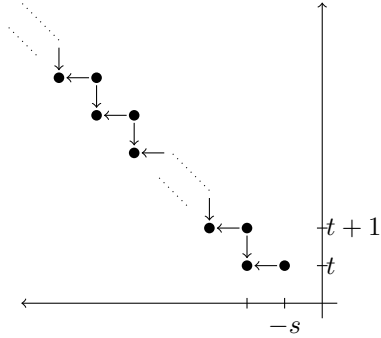
The product of all generators in total degree $t-s$ is then the only cycle which is not also a boundary, and this product lies in $F^{-s}$ but not in $F^{-s-1}$.

\end{proof}

\begin{lem}\label{L:Hofhomotopyorbits}
Let $U=\K \bkspace_{(\infty, s, t)}$. Then \[ H(\Tot (\K E\pi \tp U^{\ten 2})) \cong HTCN(\K E\pi \tp U^{\ten 2}) \cong H(\K E\pi \tp (\Tot U)^{\ten 2}).\] In particular, the $k^{\text{th}}$ graded piece of each is $\K$ for $k\geq 2(t-s)$ and $0$ for $k<2(t-s)$. Furthermore, for $m\geq t-s$, the filtration is given by
\begin{gather*} \K = F^{-v} H_{t-s+m}(\Tot (\K E\pi \tp U^{\ten 2})) = \dots = H_{t-s+m}(\Tot (\K E\pi \tp U^{\ten 2})) \\
0 = F^{-v-1} H_{t-s+m}(\Tot (\K E\pi \tp U^{\ten 2})) \end{gather*} 
where \[ v = \begin{cases} s & m\geq t \\ t+s-m & t-s \leq m \leq t. \end{cases} \]
\end{lem}
\begin{proof} 
We first compute $H(\K E\pi \tp (\Tot U)^{\ten 2})$.
Notice that the shuffle map gives an isomorphism
\[ HN(\K E\pi \tp (\Tot U)^{\ten 2}) \cong H( (N\K E\pi) \tp (N\Tot U)^{\ten 2}) = H(W\tp (N\Tot U)^{\ten 2}).\] We computed the homology of $\Tot U$ in Lemma~\ref{L:totV}. We thus have
\[ H_k(W\tp (N\Tot U)^{\ten 2}) = H_k(W\tp (HN\Tot U)^{\ten 2}) = \begin{cases} \K & k\geq 2(t-s) \\ 0 & k< 2(t-s)\end{cases} \]
by \cite[Lemmas 1.1 and 1.3]{may} using the same arguments from \cite[Section 3.1]{me1}.

The first isomorphism in the statement of the lemma comes from the quasi-isomorphism $\phi_\infty$  from \eqref{E:phiell}. The shuffle map  \[ \nabla: CN(\K E\pi \tp U^{\ten 2}) \to C(W\tp (NU)^{\ten 2}) = C(W\tp D_{\infty st}^{\ten 2})\] is an isomorphism on $E^1$ of the associated spectral sequences, hence is an isomorphism at $E^\infty$. We saw in \cite[Theorem 8.3]{me1} that the spectral sequence for the latter bicomplex converged strongly, and the same argument (using \cite[Theorems 7.1 and 10.1]{boardman}) gives that the spectral sequence associated to $CN(\K E\pi \tp U^{\ten 2})$ converges strongly. We know what the spectral sequence for $W\tp D_{\infty st}^{\ten 2}$ looks like at $E^2$ by \cite[Theorem 7.1]{me1}, and all higher differentials $\ssd^2, \ssd^3, \dots$ are zero for structural reasons. Thus we have computed that 
\[ H_k TCN(\K E\pi \tp U^{\ten 2}) = H_k TC (W\tp D_{\infty st}^{\ten 2}) = \begin{cases} \K & k\geq 2(t-s) \\ 0 & k < 2(t-s). \end{cases} \]

To prove the statement about the filtration on $H(\Tot (\K E\pi \tp U^{\ten 2}))$, we just prove the corresponding statement for the filtration on $HTCN(\K E\pi \tp U^{\ten 2})$ since these are isomorphic as filtered modules via the quasi-isomorphism $\phi_\infty$. But we just saw that we have strong convergence for the spectral sequence associated to $CN(\K E\pi \tp U^{\ten 2})$,
so, using the same maps as in the previous paragraph,
\begin{align*} (F^p / F^{p-1}) H_{q+p} TCN(\K E\pi \tp U^{\ten 2}) & \cong E^\infty_{p,q}(W\tp D_{\infty st}^{\ten 2})\\  &= \begin{cases} \K & (p,q) \in \set{-s} \times [2t, \infty) \\
\K & (p,q) \in [-2s,-s-1] \times \set{2t} \\ 0 & \text{otherwise}\end{cases} \end{align*}
by \cite[Theorem 7.1]{me1}.
\end{proof}

Most of the work of the past few sections was done to prove
\begin{prop}\label{P:interchange}
Let $U=\K \bkspace_{(\infty, s, t)}$. Then 
\[ H(\intalg): H(\K E\pi \tp (\Tot U)^{\ten 2}) \to H(\Tot (\K E\pi \tp U^{\ten 2})) \] is an isomorphism.
\end{prop}
\begin{proof}
We showed in Lemma~\ref{L:Hofhomotopyorbits} that \[H(\Tot (\K E\pi \tp U^{\ten 2})) \cong HTCN(\K E\pi \tp U^{\ten 2})\] and that these are of finite type, but mentioned nothing about any chain maps. In this proof we will show that $H(\intalg)$ is an injection, which implies that it is an isomorphism.

We now use the fact that $\Tot U$ has the homology of a sphere (Lemma~\ref{L:totV}). Write $v$ for the generator of $H_{t-s}(\Tot U)$. 
We then have, as in the previous lemma, that \begin{align*} H_k (\K E\pi \tp (H\Tot U)^{\ten 2}) &= H_k(\K E\pi \tp (\Tot U)^{\ten 2} ) \\&= \begin{cases} \K = \K ( e_{k-2t+2s} \ten v \ten v) & k\geq 2(t-s) \\ 0 & k< 2(t-s). \end{cases} \end{align*} 
This implies that $H(\K E\pi \tp (\Tot U)^{\ten 2} )$ is a cofree $H(B\pi)$-comodule, with coaction given by
\[ \coact ( e_n \ten v \ten v) = \sum_{i+j=n} \bar e_i \ten (e_j \ten v \ten v) \] 
as in Proposition~\ref{P:nzh}.
The cofreeness, combined with the fact that the cogenerator $e_0\ten v \ten v$ maps to a nonzero element under $H(\intalg)$, will imply that the comodule map  $H(\intalg)$ is an injection. We give details in the next paragraph.

By Proposition~\ref{P:tensortocommtensor} and the fact that $\phi_\infty$ is a  quasi-isomorphism, the map
\[ H(\K E\pi \tp (H\Tot U)^{\ten 2}) = H(\K E\pi \tp (\Tot U)^{\ten 2} ) \to H(\Tot(\K E\pi \tp U^{\ten 2}))  \] 
takes $e_0 \ten v \ten v$ to a nonzero element $h$, so  is an isomorphism in degree $2(t-s)$.
%on $H_{2(t-s)}$. 
Theorem~\ref{T:comodules} gives commutativity of the top square of the following diagram:
\[ \xymatrix{
H(\K E\pi \tp (\Tot U)^{\ten 2}) \ar@{->}[r]^{H\intalg} \ar@{->}[d]^{\coact_1} & H(\Tot\K E\pi \tp U^{\ten 2}) \ar@{->}[d]^{\coact_3}\\ 
HB\pi \ten H(\K E\pi \tp (\Tot U)^{\ten 2}) \ar@{->}[r]^{1\ten H\intalg} \ar@{->}[d] & HB\pi \ten H(\Tot\K E\pi \tp U^{\ten 2}) \ar@{->}[d]\\ 
HB\pi \ten H_{2(t-s)}(\K E\pi \tp (\Tot U)^{\ten 2}) \ar@{->}[r]^{1\ten H\intalg}  & HB\pi \ten H_{2(t-s)}(\Tot\K E\pi \tp U^{\ten 2}). \\ 
}\]
The unlabeled vertical arrows are just projection onto the appropriate graded piece, so the bottom square commutes as well.
Remember that we wish to show that the top map is an injection. Since $H_{n+2(t-s)} ( \K E\pi \tp (\Tot U)^{\ten 2})$ is $1$-dimensional for $n\geq 0$, we must only check that $e_n \ten v \ten v$ does not map to zero. To do this we compute its image in $H(B\pi) \ten H_{2(t-s)}(\Tot\K E\pi \tp U^{\ten 2})$ by the bottom left composite. We have
\[ e_n \ten v \ten v \mapsto \sum_{i+j = n} \bar e_i \ten (e_j \ten v \ten v) \mapsto \bar e_n \ten  (e_0 \ten v \ten v) \mapsto \bar e_n \ten h \neq 0. \]
Thus the comodule map $H(\intalg)$ is an injection.
\end{proof}

\begin{prop}\label{P:filtrquotients}
Let $U = \K \bkspace_{(\infty, s, t)}$. Then, for $m\geq t-s$, 
\begin{align*}  H_{t-s} (\Tot U) &= (F^{-s} / F^{-s-1}) H_{t-s} (\Tot U) \qquad \text{and} \\
 \ H_{t-s+m} (\Tot (\K E\pi \tp U^{\ten 2})) &= (F^{-v} / F^{-v-1}) H_{t-s+m} (\Tot (\K E\pi \tp U^{\ten 2})) \end{align*} 
where \[ v = \begin{cases} s & m\geq t \\ t+s-m & t-s \leq m \leq t. \end{cases} \] Furthermore, the modules listed above are one-dimensional vector spaces
and $Q^m(F^{-s}(H\Tot U)) \ci F^{-v}H(\Tot (\K E\pi \tp U^{\ten 2}))$.
\end{prop}
\begin{proof}
%\begin{proof}
%The statements about rank and filtration quotients are just 
%\[ H_{t-s} (\Tot V) \text{ and } H_{t-s+m} (\Tot (\K E\pi \tp V^{\ten 2})) \] are one-dimensional (the latter when $m\geq t-s$) are just Lemmas~\ref{L:totV} and~\ref{L:Hofhomotopyorbits}. The equality $H_{t-s}(\Tot V) = (F^{-s} / F^{-s-1}) H_{t-s}(\Tot V)$ is a direct consequence of Lemma~\ref{L:totV}.
%By strong convergence of the spectral sequence associated to $CN(\K E\pi \tp V^{\ten 2})$, we have that
%\[ (F^{-p} / F^{-p-1}) H_{t-s+m} (\Tot (\K E\pi \tp V^{\ten 2})) \cong E^\infty_{} =\begin{cases} \K & k\geq 2(t-s) \\ 0 & k< 2(t-s)\end{cases} \]
The first statements follow directly from Lemmas~\ref{L:totV} and~\ref{L:Hofhomotopyorbits}. The definition given for $Q^m: H(\Tot U) \to H(\Tot (\K E\pi \tp U^{\ten 2})$ in Proposition~\ref{T:geomtoalg} shows that it factors as \[ H(\Tot U) \overset{H(\nabla q^m)}{\longrightarrow} H(\K E\pi \tp (\Tot U)^{\ten 2}) \overset{\cong}{\to} H(\Tot (\K E\pi \tp U^{\ten 2})) \]
where the first map is nontrivial and the second map is an isomorphism by Proposition~\ref{P:interchange}. Thus $Q^m$ is nontrivial. We know everything about the filtrations from Lemmas~\ref{L:totV} and~\ref{L:Hofhomotopyorbits}, so we see that $Q^m(F^{-s}) \ci F^{-v}$, as desired.
\end{proof}

\begin{proof}[Proof of Theorem \ref{T:convergence}]
Let $V$ be a cosimplicial simplicial module and let $v$ be as in the statement of the theorem. 
When $V=U=\K \bkspace_{(\infty, s,t)}$, we know that the top left term and the bottom right term in the diagram \[ \xymatrix{
(F^{-s} / F^{-s-1}) H_{t-s} (\Tot V) \ar@{->}[r] \ar@{->}[dd]^{Q^m} & E_{-s,t}^\infty (NV) \ar@{->}[d]^{\ext^m} \\
& E_{-v, v-s+m+t}^\infty (W\tp NV^{\ten 2}) \ar@{->}[d]^{\cong}_{E^1 - iso} \\
(F^{-v} / F^{-v-1}) H_{t-s+m} (\Tot (\K E\pi \tp V^{\ten 2})) \ar@{->}[r] & E_{-v, v-s+m+t}^\infty (N(\K E\pi \tp V^{\ten 2}))
}\]
%commutes because
%\[ E_{-v, v-s+m+t}^\infty (N(\K E\pi \tp V^{\ten 2})) \cong \K \]
are both $\K$ by Proposition~\ref{P:filtrquotients} and \cite[Theorem 7.1]{me1}. Of course the horizontal maps are %just the abutment maps, hence 
injections, and the right hand composite is an injection by the definition of $Q^m$ on the spectral sequence level from %\cite[\textbf{p.~29}]{me1}.
\cite[Section 10]{me1}. Finally, the external operation \[ Q^m: H_{t-s}(\Tot U) \to H_{t-s+m} (\K E\pi \tp (\Tot U)^{\ten 2}) \] is nontrivial, so the $Q^m$ on the left is an injection by Propositions~\ref{P:interchange} and~\ref{P:filtrquotients}. Thus this diagram commutes.

All maps in this diagram are natural in $V$. We now show that this diagram commutes for an arbitrary cosimplicial simplicial module $V$. Let $x$ be an element of 
\[ (F^{-s} / F^{-s-1}) H_{t-s} (\Tot V) \ci E^\infty_{-s,t}(NV), \] 
which is then represented by a map $\K \bkspace_{(\infty, s, t)} \to V$. By naturality the two composites in the diagram evaluate to the same thing on the element $x$. But $x$ was arbitrary, so the diagram commutes.

Now let $X$ be a cosimplicial space. By Proposition~\ref{T:geomtoalg} the diagram 
\[ \xymatrix{ 
H (\Tot X) \ar@{->}[r] \ar@{->}[d]^{Q^m} & H(\Tot \K X) \ar@{->}[d]^{Q^m}\\
H (\Tot (E\pi \times_\pi X^{\times 2}) ) \ar@{->}[r] & H(\Tot (\K E\pi \tp (\K X)^{\ten 2})) \\
}\]
commutes. Then
\[ \xymatrix{ 
(F^{-s} / F^{-s-1}) H_{t-s} (\Tot X) \ar@{->}[r] \ar@{->}[d]^{Q^m} & (F^{-s} / F^{-s-1}) H_{t-s}(\Tot \K X) \ar@{->}[d]^{Q^m}\\
(F^{-v} / F^{-v-1}) H_{t-s+m} (\Tot (E\pi \times_\pi X^{\times 2}) ) \ar@{->}[r] & (F^{-v} / F^{-v-1}) H_{t-s+m}(\Tot (\K E\pi \tp (\K X)^{\ten 2})) \\
}\]
commutes as well. Combining this diagram with the one from the beginning of the proof, we see that the diagram in the theorem statement commutes.

Finally, we have 
\[ Q^m(F^{-s}(H\Tot V)) \ci F^{-v}H(\Tot (\K E\pi \tp V^{\ten 2}))\]
for an arbitrary cosimplicial simplicial module $V$ by Proposition~\ref{P:filtrquotients}, the fact that elements of $H(\Tot V)$ are representable, and naturality. By definition of the filtration and Proposition~\ref{T:geomtoalg} we thus have that
\[ Q^m [ F^{-s} H_{t-s} (\Tot (X))] \ci F^{-v}H_{t-s+m} (\Tot(E\pi \times_\pi X^{\times 2})) \]
for any cosimplicial space $X$.
\end{proof}

%%%%%%%%%%%%%%%%%%%%%%%%%%%%%%%%%%%%%%%%%%%%%%%%%%
%%%%%%%%%%%%%%%%%%%%%%%%%%%%%%%%%%%%%%%%%%%%%%%%%%
%%%%%%%%%%%%%%%%%%%%%%%%%%%%%%%%
%\appendix
%%%%%%%%%%%%%%%%%%%%%%%%%%%%%%%%
\section{Proof of Theorem~\ref{T:tensorproducts}}\label{S:Ptensor}
%%%%%%%%%%%%%%%%%%%%%%%%%%%%%%%%
%%%%%%%%%%%%%%%%%%%%%%%%%%%%%%%%%%%%%%%%%%%%%%%%%%
%%%%%%%%%%%%%%%%%%%%%%%%%%%%%%%%%%%%%%%%%%%%%%%%%%

In this section, we prove the following theorem, which states that in homology the map $\phi_\mess$ is compatible with tensor products.
\theoremstyle{plain}
\newtheorem*{thmfortytwo}{Theorem \ref{T:tensorproducts}}

\begin{thmfortytwo} The diagram
\[ \xymatrix{
H(\Tot_\mess U) \ten H(\Tot_\mess V) \ar@{->}[r]^-{\phi_\mess \ten \phi_\mess} \ar@{->}[d]^{H(\intb \nabla)}  &  H (T_\mess CNU) \ten H (T_\mess CNV) \ar@{->}[d]^{ H(\nabla AW)} \\
H (\Tot_\mess (U\ten V)) \ar@{->}[r]^{\phi_\mess} & H(T_\mess CN(U\ten V))
}\] commutes.
\end{thmfortytwo}

We prove this on the chain level by showing that the diagram
\[ \xymatrix{
N  \Tot_\mess U \ten N \Tot_\mess V \ar@{->}[r]^{\phi_\mess \ten \phi_\mess} 
& T_\mess CNU \ten T_\mess CNV \ar@{->}[d] 
\\ N[ \Tot_\mess U \ten  \Tot_\mess V] \ar@{->}[u]^{AW} \ar@{->}[dd]^{N(\intb)}
& T_\mess[CNU \ten CN V] \ar@{->}[d]^{T_\mess AW} 
\\ & T_\mess C[NU \ten N V]
\\ N[ \Tot_\mess (U \ten V)] \ar@{->}[r]_{\phi_\mess} 
&T_\mess CN (U\ten V)  \ar@{->}[u]_{T_\mess C(AW)}
}\]
commutes. We will do this by calculating each composite explicitly and using an elementary fact about shuffles. The result will then follow from the K\"unneth theorem and the Eilenberg-Zilber theorem.

The reader may wish to review the definitions of $\slush{f}$ from \eqref{E:funderline} and $\phi_\mess$ from \eqref{E:phiell}. We let $\shuff{p,q}$ be the set of $(p,q)$-shuffles, considered as permutations of the set $\set{0, 1, \dots, p+q-1}$.

\begin{calc}\label{CALC:upperrightfortytwo}
Let $(f) \ten (g) \in N[  \Tot_\mess (U) \ten  \Tot_\mess(V)]_q$, and consider the image of $(f)\ten (g)$ under the composite 
\[ \xymatrix{
N  \Tot_\mess U \ten N \Tot_\mess V \ar@{->}[r]^-{\phi_\mess \ten \phi_\mess} & T_\mess CNU \ten T_\mess CNV \ar@{->}[d] \\
N[ \Tot_\mess U \ten  \Tot_\mess V] \ar@{->}[u]^{AW} & T_\mess[CNU \ten CN V] \ar@{->}[d]^{T_\mess (AW)} \\
& T_\mess C[NU \ten N V].
}\] The restriction of the image to the component in cosimplicial degree $p$ with the first component of degree $z$ (i.e. the part in $T_\mess C^p[N_z U \ten NV]$) is
\[ \sum_{j+\bj=q} \sum_{\alpha, \beta \in \shuff{k,j} \times \shuff{\bk,\bj}}\left( \begin{gathered} f ( d^{k+\bk} \cdots d^{k+1}s^{\alpha(k)} \cdots s^{\alpha(k+j-1)}, d^q \cdots d^{j+1}s^{\alpha(0)} \cdots s^{\alpha(k-1)})\\ \ten g( d^{k-1} \cdots d^{0} s^{\beta(\bk)}  \cdots s^{\beta(\bk + \bj -1)},  d^{j-1} \cdots d^0 s^{\beta(0)} \cdots s^{\beta(\bk-1)})
\end{gathered} \right). \]
\end{calc}
%
%We will show that the following two composites are the same. The first is (I)
%
%and the second is (II)
%
%
%We will take tensor $(f) \ten (g) \in N[  \Tot_\mess (U) \ten  \Tot_\mess(V)]_q$ and compute the image of both (I) and (II) in cosimplicial degree $p$ and simplicial degree of the first component $z$.
\begin{proof}
Let \begin{align*}
(F) &\in (\Tot_\mess U)_j \\
(G)&\in(\Tot_\mess V)_\bj
\end{align*}
and consider the elementary tensor $(F) \ten (G) \in (N \Tot_\mess U)_j \ten (N \Tot_\mess  V)_\bj$.
We have, on the top,
\[ (F) \ten (G) \mapsto \left[ \slush{ F} \sum_{k=0}^\mess \imath_k \wedge \imath_j \right]  \ten \left[ \slush{ G} \sum_{\bk=0}^\mess \imath_\bk \wedge \imath_\bj \right]. \] 
In cosimplicial degree $p\leq \mess$, the image of this element under the Alexander-Whitney map is
\begin{multline*} \sum_{k=0}^\mess \sum_{\bk=0}^{k-p} \left[ d^{k+\bk} \cdots d^{k+1} F^k (\imath_k \wedge \imath_j) \right] \ten \left[ d^{k-1} \cdots d^0 G^{\bk} (\imath_\bk \wedge \imath_\bj) \right] \\
= \sum_{k+\bk=p} \Big[ F^{p} ((d^{k+\bk} \cdots d^{k+1} \imath_k)\wedge \imath_j) \Big] \ten \Big[ G^{p} ((d^{k-1} \cdots d^{0} \imath_\bk)\wedge \imath_\bj) \Big]. \end{multline*}
Focusing only on the elements for which the first component is in homological degree $z$, we have $k=z-j$ and $\bk=p+j-z$. 
Recall from section~\ref{S:totalization} that $\imath_k \wedge \imath_q$ is $\nabla(\imath_k\ten \imath_q)$.
Thus the previous expression becomes
\[  \sum_{\alpha \in \shuff{k,j}} \sum_{\beta \in \shuff{\bk, \bj}} \left( \begin{gathered} F^p (s_{\alpha(k+j-1)} \cdots s_{\alpha(k)} d^{k+\bk} \cdots d^{k+1} \imath_k,  s_{\alpha(k-1)} \cdots s_{\alpha(0)} \imath_j) \\
\ten G^p (s_{\beta(\bk + \bj -1)} \cdots s_{\beta(\bk)} d^{k-1} \cdots d^{0} \imath_\bk,  s_{\beta(\bk-1)} \cdots s_{\beta(0)} \imath_\bj) \end{gathered} \right) \]
where $\shuff{a,b}$ is the set of $(a,b)$-shuffles, thought of as permutations of the set $\set{0, \dots, a+b-1}$. Of course $\imath_k$ is just the identity map $[k] \to [k]$, so this is
\begin{equation*}\label{E:bigFG} \sum_{\alpha, \beta \in \shuff{k,j} \times \shuff{\bk,\bj}}\left( \begin{gathered} F ( d^{k+\bk} \cdots d^{k+1}s^{\alpha(k)} \cdots s^{\alpha(k+j-1)},s^{\alpha(0)} \cdots s^{\alpha(k-1)})\\ \ten G( d^{k-1} \cdots d^{0} s^{\beta(\bk)}  \cdots s^{\beta(\bk + \bj -1)},  s^{\beta(0)} \cdots s^{\beta(\bk-1)})
\end{gathered} \right). \end{equation*}
%Here the input for $F$ is a map $[z] = [j+k] \to [k] \to [k+\bk] = [p]$ and a map $[z]=[j+k] \to [j]$.
The domains for $F$ and $G$ are $\Delta_z^p \times \Delta_z^j$.

Now if we consider $(f) \in \Tot_\mess (U)_{q}$ and $(g) \in \Tot_\mess(V)_q$, we can apply the Alexander-Whitney map to $((f),(g))$ and get
\[ \sum_{j+\bj = q} \big((f \circ (\id_{\sk}\times d^q \cdots d^{j+1})),  (g \circ (\id_{\sk} \times d^{j-1} \cdots d^0))\big) \]
where we are precomposing $f$ with the map $\sk_\mess \Delta \times \Delta^q \to \sk_\mess \Delta \times \Delta^j$ and similarly for $g$.
Precomposing with $AW$, equation (\ref{E:bigFG}) becomes
\begin{equation*}\label{E:smallfg} \sum_{j+\bj=q} \sum_{\alpha, \beta \in \shuff{k,j} \times \shuff{\bk,\bj}}\left( \begin{gathered} f ( d^{k+\bk} \cdots d^{k+1}s^{\alpha(k)} \cdots s^{\alpha(k+j-1)}, d^q \cdots d^{j+1}s^{\alpha(0)} \cdots s^{\alpha(k-1)})\\ \ten g( d^{k-1} \cdots d^{0} s^{\beta(\bk)}  \cdots s^{\beta(\bk + \bj -1)},  d^{j-1} \cdots d^0 s^{\beta(0)} \cdots s^{\beta(\bk-1)})
\end{gathered} \right). \end{equation*}

\end{proof}

%There is, of course, a bijection of sets
%\[ \coprod_{j=z-\min(z,p)}^z \shuff{z-j,j} \times \shuff{p+j-z,q-j} \overset{\cong}{\to} \shuff{p,q} \]
%obtained by letting $\shuff{z-j,j}$ act on $\set{0, 1, \dots, z-1}$ and $\shuff{p+j-z,q-j}$ act on $\set{z, \dots, p+q-1}$. Given the conditions above we have $z-j=k$, $p+j-z=\bk$, and $q-j=\bj$.
%\textbf{COME BACK HERE AFTER DOING THE OTHER WAY}

\begin{calc}\label{CALC:lowerfortytwo}
Let $(f) \ten (g) \in N[  \Tot_\mess (U) \ten  \Tot_\mess(V)]_q$, and consider the image of $(f)\ten (g)$ under the composite 
\[ \xymatrix{
N[ \Tot_\mess U \ten \Tot_\mess V] \ar@{->}[d]^{N(\intb)} &  T_\mess C[NU \ten NV] \\
N[ \Tot_\mess (U \ten V)] \ar@{->}[r]^{\phi_\mess} &T_\mess CN (U\ten V)  \ar@{->}[u]^{T_\mess C (AW)}.
}\]
 The restriction of the image to the component in cosimplicial degree $p$ with the first component of degree $z$ (i.e. the part in $T_\mess C^p[N_z U \ten NV]$) is
 \[
\sum_{\gamma \in \shuff{p,q}} \left( \begin{gathered}
f( s^{\gamma(p)} \cdots s^{\gamma(p+q-1)}d^{p+q} \cdots d^{z+1},  s^{\gamma(0)} \cdots s^{\gamma(p-1)} d^{p+q} \cdots d^{z+1}) \\
\ten g(s^{\gamma(p)} \cdots s^{\gamma(p+q-1)} d^{z-1} \cdots d^0,  s^{\gamma(0)} \cdots s^{\gamma(p-1)} d^{z-1} \cdots d^0)
 \end{gathered} \right).
 \]
\end{calc}
\begin{proof}

%We now turn our attention to (II). 
Let $((f),(g))$ be as in the statement. Applying $\phi_\mess$ and the interchange map $\intb$ from \eqref{E:intb}, we obtain
\[ \sum_{l=0}^\mess \slush{\intb(f,g)} (\imath_l \wedge \imath_q). \] The part in cosimplicial degree $p$ is
\[  \slush{\intb(f,g) } (\imath_p \wedge \imath_q) = \sum_{\gamma \in \shuff{p,q}} \left( \begin{gathered}
f(s_{\gamma(p+q-1)} \cdots s_{\gamma(p)} \imath_p, s_{\gamma(p-1)} \cdots s_{\gamma(0)} \imath_q) \\
\ten g(s_{\gamma(p+q-1)} \cdots s_{\gamma(p)} \imath_p, s_{\gamma(p-1)} \cdots s_{\gamma(0)} \imath_q)
 \end{gathered} \right). \]
Applying $d_{z+1} d_{z+2} \cdots d_{p+q} \ten d_0 \cdots d_{z-1}$ amounts to applying the Alexander-Whitney map and then projecting onto the component where the first term is in homological degree $z$. The result is then
\[ d_{z+1} \cdots d_{p+q} \ten d_0 \cdots d_{z-1} \sum_{\gamma \in \shuff{p,q}} \left( \begin{gathered}
f( s^{\gamma(p)} \cdots s^{\gamma(p+q-1)} ,  s^{\gamma(0)} \cdots s^{\gamma(p-1)}) \\
\ten g(s^{\gamma(p)} \cdots s^{\gamma(p+q-1)} ,  s^{\gamma(0)} \cdots s^{\gamma(p-1)})
 \end{gathered} \right) \]
\begin{equation*}\label{E:othercomposite}
=\sum_{\gamma \in \shuff{p,q}} \left( \begin{gathered}
f( s^{\gamma(p)} \cdots s^{\gamma(p+q-1)}d^{p+q} \cdots d^{z+1},  s^{\gamma(0)} \cdots s^{\gamma(p-1)} d^{p+q} \cdots d^{z+1}) \\
\ten g(s^{\gamma(p)} \cdots s^{\gamma(p+q-1)} d^{z-1} \cdots d^0,  s^{\gamma(0)} \cdots s^{\gamma(p-1)} d^{z-1} \cdots d^0)
 \end{gathered} \right).
 \end{equation*}

\end{proof}

We must show that the two calculations give the same result. 
%We must show that expression~(\ref{E:smallfg}) is equal to this. 
To do so, we use the following version of Vandermonde's identity.

\begin{prop}[Shuffle Splitting]\label{P:shuffbijection} Fix $p, q \geq 0$ and $z\in (0,p+q)$ integers.
There is a bijection of sets
\[ \eta: \coprod_{k=\max (0,z-q)}^{\min(p,z)} \shuff{k,z-k} \times \shuff{p-k,q+k-z} \to \shuff{p,q} \] given, for $(\alpha, \beta) \in \shuff{k,z-k} \times \shuff{p-k,q+k-z}$, by 
\[
%\alpha, \beta & \mapsto \gamma \\
\eta(\alpha, \beta)(c)= \begin{cases}
\alpha(c) & c\in [0,k-1] \\
\beta(c-k) + z & c\in [k,p-1] \\
\alpha(c-p+k) & c\in [p,p+z-k-1] \\
\beta(c-z) + z & c\in [p+z-k,p+q-1].
\end{cases}
\]\end{prop}
\begin{proof}
One checks that $\gamma = \eta (\alpha, \beta)$ is a $(p,q)$-shuffle and that  \[ \eta_k: \shuff{k,z-k} \times \shuff{p-k,q+k-z} \to \shuff{p,q} \] is injective for all $k$. Furthermore, if $\gamma = \eta_k (\alpha, \beta)$ then we can recover  $k$ as $\min \gamma^\inv [z,p+q-1]$,  so $\eta =\coprod \eta_k$ is an injection as well.
%
%Now we have an injection of sets
%\[ \eta = \sqcup \eta_k: \coprod_{k=\max (0,z-q)}^{\min(p,z)} \shuff{k,z-k} \times \shuff{p-k,q+k-z} \hookrightarrow \shuff{p,q}. \]
Now we can see that $\eta$ is a bijection by comparing sizes of sets. This follows from  the fact that $|\shuff{p,q}| = \binom{p+q}{p}$ and Vandermonde's identity
\[ \sum_{k=\max (0,z-q)}^{\min(p,z)} \binom{z}{k}\binom{p+q-z}{p-k} = \sum_{k=0}^z \binom{z}{k}\binom{p+q-z}{p-k}  = \binom{p+q}{p}.\]
\end{proof}

\begin{proof}[Proof of Theorem~\ref{T:tensorproducts}]
%Consider the composite $s^{a_0} \cdots s^{a_{n-1}}$ considered as an ordered map $[w] \to [w-n]$, where  $a_0 < a_1 < \cdots < a_{n-1} < w$.
%It is given by the formula 
%\[
%s^{a_0} \cdots s^{a_{n-1}} (c) = \begin{cases} c & c\leq a_0 \\ c-t & c\in (a_{t-1}, a_t] \\ c-n & c> a_{n-1}. \end{cases}
%\]

Isolate $(\alpha, \beta) \in \shuff{k,j} \times \shuff{\bk, \bj}$ and let $\gamma = \eta(\alpha, \beta)$. When applying other results from this section, we set $z=j+k$, $p= k+\bk$, and $q=j+\bj$. 
In light of Proposition~\ref{P:shuffbijection}, to show that the results from calculations~\ref{CALC:upperrightfortytwo} and \ref{CALC:lowerfortytwo} are equal, we must show the following equalities of ordered maps originating from $[z]$:
\begin{align*}
s^{\gamma(p)} \cdots s^{\gamma(p+q-1)}d^{p+q} \cdots d^{z+1} &= d^{k+\bk} \cdots d^{k+1}s^{\alpha(k)} \cdots s^{\alpha(k+j-1)}\\
s^{\gamma(0)} \cdots s^{\gamma(p-1)} d^{p+q} \cdots d^{z+1} &= d^q \cdots d^{j+1}s^{\alpha(0)} \cdots s^{\alpha(k-1)}\\
s^{\gamma(p)} \cdots s^{\gamma(p+q-1)} d^{z-1} \cdots d^0 &= d^{k-1} \cdots d^{0} s^{\beta(\bk)}  \cdots s^{\beta(\bk + \bj -1)} \\
s^{\gamma(0)} \cdots s^{\gamma(p-1)} d^{z-1} \cdots d^0 &= d^{j-1} \cdots d^0 s^{\beta(0)} \cdots s^{\beta(\bk-1)}.
\end{align*}
The sequence $\gamma(p), \cdots, \gamma(p+q-1)$ is, by definition of $\eta$, \[ \alpha(k), \dots,  \alpha(k+j-1), \beta(\bk)+z, \dots, \beta(\bk + \bj -1) + z. \] This establishes the first and third equalities.  The sequence $\gamma(0), \dots, \gamma(p-1)$ is
\[ \alpha(0), \dots, \alpha(k-1), \beta(0)+z, \dots, \beta(\bk-1) + z,\]
which establishes the second and fourth equalities.

%
%To see that expressions~(\ref{E:smallfg}) and (\ref{E:othercomposite}) are equal,
%consider $(\alpha, \beta) \in \shuff{k,j} \times \shuff{\bk, \bj}$ and let $\gamma = \eta(\alpha, \beta)$. One shows the following equalities
%
%which is easy, but tedious, to do by evaluating using the formula
%
%whenever .
Hence calculations~\ref{CALC:upperrightfortytwo} and \ref{CALC:lowerfortytwo} give the same result, which completes the proof.
\end{proof}

%%%%%%%%%%%%%%%%%%%%%%%%%%%%%%%%%%%%%%%%%%%%%%%%%%
%%%%%%%%%%%%%%%%%%%%%%%%%%%%%%%%%%%%%%%%%%%%%%%%%%
\section{Proof of Theorem~\ref{T:comodules}}\label{S:prooftcomodules}
%%%%%%%%%%%%%%%%%%%%%%%%%%%%%%%%%%%%%%%%%%%%%%%%%%
%%%%%%%%%%%%%%%%%%%%%%%%%%%%%%%%%%%%%%%%%%%%%%%%%%

We now return to the proof of

\theoremstyle{plain}
\newtheorem*{thmfiftyfour}{Theorem \ref{T:comodules}}

\begin{thmfiftyfour} Let $V$ be a cosimplicial simplicial $\K$-module. Then \[
\phi_\infty N(\intalg): N(\K E\pi \tp (\Tot(V))^{\ten 2}) \to  TCN  (\K E\pi \tp V^{\ten 2})) \] is a map of $\bw$-comodules.
\end{thmfiftyfour}

Take an elementary tensor $e\ten f \ten g \in \K E\pi \tp (\Tot V)^{\ten 2}$, where $e$, $f$, and $g$ are maps
\begin{align*}
e:& \Delta^p \to \K E\pi \\
f:& \Delta^\bullet \times \Delta^p  \to V \\
g:& \Delta^\bullet \times \Delta^p  \to V. \\
\end{align*}
Recall that
\[ N_p(\intalg): N_p(\K E\pi \tp (\Tot(V))^{\ten 2}) \to N_p \Tot (\K E\pi \tp V^{\ten 2})) \]
comes from
\[ \Delta^\bullet \times \Delta^p \overset{\text{diag}}{\longrightarrow} (\Delta^\bullet \times \Delta^p )^{\times 3} \twoheadrightarrow \Delta^p \times (\Delta^\bullet \times \Delta^p )^{\times 2},\]
so that $\intalg (e\ten f \ten g)$ is a map $\Delta^\bullet \times \Delta^p  \to\K E\pi \tp V^{\ten 2} $ which sends
\[ (\faked_1 , \faked_2) \mapsto e(\faked_2) \ten f(\faked_1 , \faked_2) \ten g(\faked_1 , \faked_2). \]

We wish to show that the diagram
\[ \scalebox{.80}{ 
\xymatrix{
N(\K E\pi \tp (\Tot(V))^{\ten 2}) \ar@{->}[d]^{\coact_2} \ar@{->}[r]^{N\intalg} & N \Tot (\K E\pi \tp V^{\ten 2})) \ar@{->}[r]^{\phi_\infty} & TCN (\K E\pi \tp V^{\ten 2})) \ar@{->}[d]^{\coact_3} \\
\bw \ten N (\K E\pi \tp (\Tot(V))^{\ten 2}) \ar@{->}[r]^{1\ten N\intalg} & \bw \ten N \Tot (\K E\pi \tp V^{\ten 2})) \ar@{->}[r]^{1\ten \phi_\infty} & \bw \ten TCN (\K E\pi \tp V^{\ten 2}))
} 
}
\]
commutes. Write $\bar{e}$ for the composite $\Delta^p \to \K E\pi \to \K B\pi$ and begin with the bottom left. 

%BOTTOM COMPOSITE

\begin{calc}\label{CALC:fiftyfour}
Let $e, f, g$ be $p$-simplices as above, and consider the image of $e\ten f \ten g$ under $(1\ten \phi_\infty) \circ (1 \ten N(\intalg)) \circ \coact_2$ in $\bw \ten TCN(\K E\pi \tp V^{\ten 2})$. Restricting to the piece in  \[ \bw_r \ten C^k N_{p+k-r}(\K E\pi \tp V^{\ten 2}) \cong C^k N_{p+k-r}(\K E\pi \tp V^{\ten 2}) \] for some $r\leq p$, we have that this is equal to
\[ \sum_{\alpha\in \shuff{k,p-r}} \left( \begin{gathered}
e(d^{r-1} \cdots d^0 s_{\alpha(k-1)} \cdots s_{\alpha(0)} \imath_{p-r} ) \\
\ten f^k (s_{\alpha(p+k-r-1)} \cdots s_{\alpha(k)} \imath_k, d^{r-1} \cdots d^0 s_{\alpha(k-1)} \cdots s_{\alpha(0)} \imath_{p-r} ) \\
\ten g^k (s_{\alpha(p+k-r-1)} \cdots s_{\alpha(k)} \imath_k, d^{r-1} \cdots d^0 s_{\alpha(k-1)} \cdots s_{\alpha(0)} \imath_{p-r} )
 \end{gathered} \right).\]
\end{calc}
\begin{proof}
Applying $\coact_2$ from \eqref{E:act2}, we get
\[ e\ten f \ten g \mapsto \sum_{r+s=p} d_{r+1} \cdots d_{p} \bar{e} \ten d_0 \cdots d_{r-1} (e\ten f \ten g). \] The bottom composite is the identity on the $\bw$ component, so at this point we restrict to the part in $\bw_{r}$ for a fixed $r \leq p$. We are thus only looking at
\[ d_{r+1} \cdots d_{p} \bar{e} \ten d_0 \cdots d_{r-1} (e\ten f \ten g), \]
%(e \ten f \ten g) \circ (1 \times d^{r-1} \cdots d^0)
and we write down
\[ \phi_\infty \intalg [ d_0 \cdots d_{r-1} (e\ten f \ten g)] = \phi_\infty [ \intalg (e \ten f \ten g ) \circ (1 \times d^{r-1} \cdots d^0)]. \]
Recall from \eqref{E:funderline} that if $f\in (\Tot V)_p$, then $\slush{f}$ is the induced  map
\[ \slush{f}: TCN(\K \Delta^\bullet \ten \Delta^p) \to TCN(V).\]  By definition of $\phi_\infty$, we have \[ \phi_\infty [ \intalg (e \ten f \ten g ) \circ (1 \times d^{r-1} \cdots d^0)]  =
\sum_k \slush{ \intalg (e \ten f \ten g ) \circ (1 \times d^{r-1} \cdots d^0)}  \imath_k \wedge \imath_{p-r}. \]
We now restrict to a fixed cosimplicial degree $k$. The element
\[ \slush{ \intalg (e \ten f \ten g ) \circ (1 \times d^{r-1} \cdots d^0)}  \imath_k \wedge \imath_{p-r} \]
is
\begin{multline*}
\slush{ \intalg (e \ten f \ten g ) \circ (1 \times d^{r-1} \cdots d^0)}  \nabla (\imath_k \ten \imath_{p-r}) \\
= \slush{ \intalg (e \ten f \ten g ) \circ (1 \times d^{r-1} \cdots d^0)} \sum_{\alpha} s_{\alpha(p+k-r-1)} \cdots s_{\alpha(k)} \imath_k \ten s_{\alpha(k-1)} \cdots s_{\alpha(0)} \imath_{p-r} \\
= \sum_\alpha \left( \begin{gathered}
e(d^{r-1} \cdots d^0 s_{\alpha(k-1)} \cdots s_{\alpha(0)} \imath_{p-r} ) \\
\ten f^k (s_{\alpha(p+k-r-1)} \cdots s_{\alpha(k)} \imath_k, d^{r-1} \cdots d^0 s_{\alpha(k-1)} \cdots s_{\alpha(0)} \imath_{p-r} ) \\
\ten g^k (s_{\alpha(p+k-r-1)} \cdots s_{\alpha(k)} \imath_k, d^{r-1} \cdots d^0 s_{\alpha(k-1)} \cdots s_{\alpha(0)} \imath_{p-r} )
 \end{gathered} \right),
\end{multline*}
where $\alpha$ runs among the set of $(k, p-r)$-shuffles. 
\end{proof}

\begin{rem} 
For a fixed $\alpha$, the important data  are the arguments to $e$, $f$, and $g$. We know $e$ only takes one argument, which we could write as
\begin{equation}\label{E:argumenttoe}
d^{r-1} \cdots d^0 s_{\alpha(k-1)} \cdots s_{\alpha(0)} (\imath_{p-r}) = d^{r-1} \cdots d^0 s^{\alpha(0)} \cdots s^{\alpha(k-1)} \in \Delta_{p-r+k}^p
\end{equation}
while $f$ and $g$ have an additional argument
\begin{equation}\label{E:argumenttofg}
s_{\alpha(p+k-r-1)} \cdots s_{\alpha(k)} \imath_k = s^{\alpha(k)} \cdots s^{\alpha(p+k-r-1)} \in \Delta_{p-r+k}^k.
\end{equation}
We will return to these later.
\end{rem}

%%% IDENTIFY THE ARGUMENTS!

%TOP COMPOSITE

\begin{proof}[Proof of Theorem~\ref{T:comodules}]
We begin with a calculation of the composite $\coact_3 \circ \phi_\infty \circ \intalg$. Applying $\phi_\infty \intalg$ to $e\ten f \ten g$ and examining the part in cosimplicial degree $k$ yields
\[ \slush{\intalg(e\ten f \ten g)} \imath_k \wedge \imath_p \\
= \sum_{\mu} \left( \begin{gathered} e(s_{\mu(k-1)} \cdots s_{\mu(0)} \imath_p) \\
\ten f^k(s_{\mu(p+k-1)} \cdots s_{\mu(k)} \imath_k, s_{\mu(k-1)} \cdots s_{\mu(0)} \imath_p ) \\
\ten g^k(s_{\mu(p+k-1)} \cdots s_{\mu(k)} \imath_k, s_{\mu(k-1)} \cdots s_{\mu(0)} \imath_p ) \end{gathered}\right), \]
where $\mu$ runs over the set of $(k,p)$ shuffles. Noting that
\[ e(s_{\mu(k-1)} \cdots s_{\mu(0)} \imath_p) =  e\circ s^{\mu(0)} \cdots s^{\mu(k-1)} \] when thought of as a map
\[ \Delta^{p+k} \to \K E\pi, \]
we apply $\coact_3$ from \eqref{E:act3} and the result is
\[ \sum_{\mu} \sum_r  d_{r+1} \cdots d_{p+k} \overline{ e s^{\mu(0)} \cdots s^{\mu(k-1)}}\ten d_0 \cdots d_{r-1}  \left(    e s^{\mu(0)} \cdots s^{\mu(k-1)}
\ten f^k(-)
\ten g^k(-) \right). \]
For a fixed $(k,p)$-shuffle $\mu$, we can examine the arguments to $e$, $f$, and $g$, noting that the action of $d_0 \cdots d_{r-1}$ is as follows
\[ \xymatrix{ \K (\Delta^k \times \Delta^p)_{k+p} \ar@{->}[d]^{d_0 \cdots d_{r-1}} \ar@{->}[r]^-{f^k \text{ or } g^k} & V_{k+p}^k \ar@{->}[d]^{d_0 \cdots d_{r-1}}  \\
\K (\Delta^k \times \Delta^p)_{k+p-r} \ar@{->}[r] & V_{k+p-r}^k} \]
and similarly for $e$. The argument to $e$ is then
\begin{equation}\label{E:Targumenttoe}
d_0 \cdots d_{r-1} s^{\mu(0)} \cdots s^{\mu(k-1)} = s^{\mu(0)} \cdots s^{\mu(k-1)} d^{r-1} \cdots d^0  \in \Delta_{p-r+k}^p
\end{equation}
while $f$ and $g$ have an additional argument
\begin{equation}\label{E:Targumenttofg}
d_0 \cdots d_{r-1} s_{\mu(p+k-1)} \cdots s_{\mu(k)} (\imath_k) = s^{\mu(k)} \cdots s^{\mu(p+k-1)} d^{r-1} \cdots d^0  \in \Delta_{p-r+k}^k.
\end{equation}

To compare this calculation to calculation~\ref{CALC:fiftyfour}, 
we can cut down the set of shuffles from $\shuff{k,p}$ to $\shuff{k,p-r}$.  Notice that
\[ d_{r+1} \cdots d_{p+k} \overline{ e s^{\mu(0)} \cdots s^{\mu(k-1)}} = \overline{e} s^{\mu(0)} \cdots s^{\mu(k-1)} d^{p+k} \cdots d^{r+1} \] will often be degenerate. We have
\[ s^{\mu(0)} \cdots s^{\mu(k-1)} d^{p+k} \cdots d^{r+1} (i) = \begin{cases} i & i \leq \mu(0) \\
i-t & i\in (\mu(t-1), \mu(t)] \\
i-k & i > \mu(k-1), \end{cases} \]
so the function \[ s^{\mu(0)} \cdots s^{\mu(k-1)} d^{p+k} \cdots d^{r+1}: [r] \to [p+k] \to [p] \]
is injective precisely when $r \leq \mu(0)$. If this is not injective then
\[ \overline{e} s^{\mu(0)} \cdots s^{\mu(k-1)} d^{p+k} \cdots d^{r+1} \] is degenerate, hence zero in $\bw$.

We thus only need to consider $(k,p)$-shuffles with $r \leq \mu(0)$, so that $r\leq p$ and
\[ \mu(k) = 0, \mu(k+1) = 1, \dots, \mu(k+r-1) = r-1. \]
The set of $(k,p)$ shuffles having this property is in bijection with $\shuff{k,p-r}$. The bijection
\[ \booeta: \setm{\mu\in \shuff{k,p}}{r \leq \mu(0)} \to \shuff{k,p-r} \] is given by
\[ \booeta(\mu) (t) = \begin{cases} \mu(t) -r & t \in [0 , k-1] \\
\mu(t+r)-r & t \in [ k, k+p-r-1] \end{cases} \]
and
\[ \booeta^\inv (\alpha) (t) = \begin{cases} \alpha(t) +r & t \in [0, k-1] \\
t-k & t\in [k, k+r -1] \\
\alpha(t-r) + r & t\in[ k+r , k+p-1]. \end{cases} \]

But then, if $\alpha = \booeta(\mu)$, the sequence $\alpha(0), \dots, \alpha(k-1)$ is equal to 
\[ \mu(0) - r, \dots, \mu(k-1) -r\]
and the sequence $\alpha(k), \dots, \alpha(p+k-r-1)$ is equal to
\[ \mu(k+r) -r, \dots, \mu(p+k-1)-r,\]
both from the definition of $\booeta$. This implies the following identities:
\begin{align}\label{E:sddseta}
s^{\mu(0)} \cdots s^{\mu(k-1)} d^{r-1} \cdots d^0 &= d^{r-1} \cdots d^0 s^{\alpha(0)} \cdots s^{\alpha(k-1)}  
\\ 
\label{E:sdseta}
s^{\mu(k)} \cdots s^{\mu(p+k-1)} d^{r-1} \cdots d^0 &= s^{\alpha(k)} \cdots s^{\alpha(p+k-r-1)}.
\end{align}

Using these for a fixed $\alpha$ and $\mu = \booeta^\inv(\alpha)$, (\ref{E:sddseta}) shows that (\ref{E:Targumenttoe}) = (\ref{E:argumenttoe}) while (\ref{E:sdseta}) shows that (\ref{E:Targumenttofg}) = (\ref{E:argumenttofg}). The arguments for the $\bw$ part, $\overline{e}$, are the same in each, so we have shown, as desired, that
\[ \coact_3 \phi_\infty N(\intalg) = (1\ten \phi_\infty) (1\ten N(\intalg)) \coact_2. \]

\end{proof}

\section*{Acknowledgements}
The author thanks Jim McClure for helpful discussions and suggestions concerning this paper, including the suggestion to think about this problem in the first place. 
The author thanks the referee for detailed and extensive feedback, which helped the author improve the paper considerably.
The author also thanks Sean Tilson for useful comments and questions.

\bibliographystyle{plain}
\bibliography{operations}

\begin{thebibliography}{10}

\bibitem{bauesmuro}
H.-J. Baues and F.~Muro.
\newblock Cohomologically triangulated categories. {II}.
\newblock {\em J. K-Theory}, 3(1):1--52, 2009.

\bibitem{boardman}
J.~Michael Boardman.
\newblock Conditionally convergent spectral sequences.
\newblock In {\em Homotopy invariant algebraic structures ({B}altimore, {MD},
  1998)}, volume 239 of {\em Contemp. Math.}, pages 49--84. Amer. Math. Soc.,
  Providence, RI, 1999.

\bibitem{bousfield}
A.~K. Bousfield.
\newblock On the homology spectral sequence of a cosimplicial space.
\newblock {\em Amer. J. Math.}, 109(2):361--394, 1987.

\bibitem{bk}
A.~K. Bousfield and D.~M. Kan.
\newblock A second quadrant homotopy spectral sequence.
\newblock {\em Trans. Amer. Math. Soc.}, 177:305--318, 1973.

\bibitem{gj}
Paul~G. Goerss and John~F. Jardine.
\newblock {\em Simplicial Homotopy Theory}, volume 174 of {\em Progress in
  Mathematics}.
\newblock Birkh\"auser Verlag, Basel, 1999.

\bibitem{me2b}
Philip Hackney.
\newblock Homology operations and cosimplicial iterated loop spaces.
\newblock To appear: Homology, Homotopy Appl.,
  \href{http://arxiv.org/abs/1102.0020}{arXiv:1102.0020} [math.AT].

\bibitem{me1}
Philip Hackney.
\newblock Operations in the homology spectral sequence of a cosimplicial
  infinite loop space.
\newblock {\em J. Pure Appl. Algebra}, 2012.
\newblock
  \href{http://dx.doi.org/10.1016/j.jpaa.2012.10.002}{doi:10.1016/j.jpaa.2012.10.002}
  (in press).

\bibitem{kudoaraki}
Tatsuji Kudo and Sh{\^o}r{\^o} Araki.
\newblock Topology of {$H_n$}-spaces and {$H$}-squaring operations.
\newblock {\em Mem. Fac. Sci. Ky\=usy\=u Univ. Ser. A.}, 10:85--120, 1956.

\bibitem{homology}
Saunders Mac~Lane.
\newblock {\em Homology}.
\newblock Die Grundlehren der mathematischen Wissenschaften, Bd. 114. Academic
  Press Inc., Publishers, New York, 1963.

\bibitem{may}
J.~Peter May.
\newblock A general algebraic approach to {S}teenrod operations.
\newblock In {\em The {S}teenrod {A}lgebra and its {A}pplications ({P}roc.
  {C}onf. to {C}elebrate {N}. {E}. {S}teenrod's {S}ixtieth {B}irthday,
  {B}attelle {M}emorial {I}nst., {C}olumbus, {O}hio, 1970)}, Lecture Notes in
  Mathematics, Vol. 168, pages 153--231. Springer, Berlin, 1970.

\bibitem{mayso}
J.~Peter May.
\newblock {\em Simplicial Objects in Algebraic Topology}.
\newblock Chicago Lectures in Mathematics. University of Chicago Press,
  Chicago, IL, 1992.
\newblock Reprint of the 1967 original.

\bibitem{turner}
James~M. Turner.
\newblock Operations and spectral sequences. {I}.
\newblock {\em Trans. Amer. Math. Soc.}, 350(9):3815--3835, 1998.

\bibitem{weibel}
Charles~A. Weibel.
\newblock {\em An Introduction to Homological Algebra}, volume~38 of {\em
  Cambridge Studies in Advanced Mathematics}.
\newblock Cambridge University Press, Cambridge, 1994.

\end{thebibliography}
\end{document}